\documentclass{article}
\usepackage{amsfonts}
\usepackage{amssymb}
\usepackage{amsmath}
\usepackage{mathtools}
\usepackage{mathrsfs}
\usepackage{xcolor}
\usepackage[utf8]{inputenc}
\usepackage[margin=1.5in]{geometry}
\usepackage[T1]{fontenc}
\usepackage{amsthm}
\usepackage[english]{babel}
\usepackage{graphicx}
\usepackage{tikz-cd}
\usepackage{tikz}
\usepackage{rotating}
\usepackage{esint}
\usetikzlibrary{matrix, arrows}

\newtheorem{theorem}{Theorem}[subsection]
\newtheorem{prop}[theorem]{Proposition}
\newtheorem{conjecture}[theorem]{Conjecture}
\newtheorem{heuristic}[theorem]{Heuristic}
\newtheorem{corollary}[theorem]{Corollary}
\theoremstyle{definition}
\newtheorem{example}[theorem]{Example}
\theoremstyle{definition}
\newtheorem{remark}[theorem]{Remark}
\theoremstyle{definition}
\newtheorem{lemma}[theorem]{Lemma}
\theoremstyle{definition}
\newtheorem{defi}[theorem]{Definition}

\newcommand\bsni{\bigskip\noindent}

\newcommand\bbc{\mathbb{C}}
\newcommand\bbd{\mathbb{D}}

\newcommand\bbp{\mathbb{P}}

\newcommand\bbr{\mathbb{R}}
\newcommand\bbs{\mathbb{S}}
\newcommand\cA{\mathcal{A}}

\newcommand\cC{\mathcal{C}}

\newcommand\cE{\mathcal{E}}

\newcommand\cH{\mathcal{H}}

\newcommand\cJ{\mathcal{J}}
\newcommand\cL{\mathcal{L}}
\newcommand\cM{\mathcal{M}}
\newcommand\cN{\mathcal{N}}
\newcommand\cO{\mathcal{O}}
\newcommand\cP{\mathcal{P}}
\newcommand\cR{\mathcal{R}}

\newcommand\cX{\mathcal{X}}
\newcommand\cY{\mathcal{Y}}
\newcommand\cZ{\mathcal{Z}}
\newcommand\K{\mathrm{K}}
\newcommand\R{\mathrm{R}}
\newcommand\mi{^{-1}}

\newcommand\an{^{\mathrm{an}}}

\newcommand\pdisc{\overline\bbd^*}
\newcommand\disc{\overline\bbd}

\DeclareMathOperator{\ord}{ord}

\DeclareMathOperator{\Specop}{Spec}

\newcommand\Spec[1]{\Specop \mathrm{#1}}

\newcommand\norm{\|\cdot\|}

\newcommand\PSH{\mathrm{PSH}}
\newcommand\MA{\,\mathrm{MA}}

\newcommand\refmetric{\phi_{\mathrm{ref}}}
\newcommand\na{^{\mathrm{NA}}}

\newcommand\path{\mathfrak{P}}

\newcommand\ideal{\mathfrak{a}}

\author{REBOULET Rémi}
\title{The space of finite-energy metrics over a degeneration of complex manifolds.}
\begin{document}
\maketitle

\begin{abstract}
\noindent Given a degeneration of compact projective complex manifolds $X\to \bbd^*$ with meromorphic singularities, and a relatively ample line bundle $L$ on $X$, we study spaces of plurisubharmonic metrics on $L$, with particular focus on (relative) finite-energy conditions. We endow the space $\hat \cE^1(L)$ of relatively maximal, relative finite-energy metrics with a $d_1$-type distance given by the Lelong number at zero of the collection of fibrewise Darvas $d_1$-distances. We show that this metric structure is complete and geodesic. Seeing $X$ and $L$ as schemes $X_\K$, $L_\K$ over the discretely-valued field $\K=\bbc((t))$ of complex Laurent series, we show that the space $\cE^1(L_\K\an)$ of non-Archimedean finite-energy metrics over $L_\K\an$ embeds isometrically and geodesically into $\hat \cE^1(L)$, and characterize its image. This generalizes previous work of Berman-Boucksom-Jonsson, treating the trivially-valued case. We investigate consequences regarding convexity of non-Archimedean functionals.
\end{abstract}

\tableofcontents

\section*{Introduction.}

\paragraph{Overview.}

The study of rays of plurisubharmonic metrics has been, for the past decade, a central component in successful approaches to important conjectures in complex geometry. Given a polarized compact Kähler manifold $(X,L)$, we define a psh ray to be a psh metric on $L\times \pdisc$ over the trivial product $X\times\pdisc$, which is furthermore invariant under the usual action of $\bbs^1$. Setting the variable $t=-\log|z|$, we obtain a "psh curve"
$$[0,\infty)\ni t\mapsto \phi_t,$$
and the study of the asymptotics of this curve (and of various quantities it can be evaluated against) plays a crucial role in many questions. A particular striking application has been an alternative proof of the Kähler-Einstein case of the Yau-Tian-Donaldson conjecture due to Berman-Boucksom-Jonsson (\cite{bbj}), and further advances in the general cscK case by Li (\cite{liytd}), both of which are major inspirations for this current work. In parallel, Darvas-Lu have started a very interesting investigation of the metric properties of the space of such rays (\cite{darlurays}), which has found many interesting developments. 

\bsni One side of the story, very apparent in the aforementioned proof of the Yau-Tian-Donaldson conjecture, is the introduction of valuative or non-Archimedean techniques to study asymptotics in Kähler geometry. Namely, to a psh ray $t\mapsto \phi_t$ as above (assuming a condition of maximality which roughly asserts that it bounds by above all other psh segments with the same endpoint values on all compact sets of the positive half-line), one can associate a non-Archimedean function $\phi\na$, defined on the Berkovich analytification $X\an$ of $X$ with respect to the trivial absolute value on $\bbc$. Without going into too much detail, $X\an$ can be essentially described as the space of valuations on $\bbc(X)$, which is then compactified using semivaluations, so that $\phi\na$ can be thought of as capturing the asymptotic singularities of $\phi$ along all possible birational models of $X$. A beautiful result is that a class of distinguished geodesic rays (the ones of interest for K-stability!) are in one-to-one correspondance with finite-energy non-Archimedean functions on $X\an$.

\bsni In the present work, we generalize such results to the wider setting of psh metrics on arbitrary families over the punctured disc. Already in the isotrivial case without $\bbs^1$-action, these spaces have garnered interest (see e.g. \cite{donaldsondiscs}, \cite{rwndiscs}), but our results also hold for arbitrary degenerations with "meromorphic" singularities. This setting is of interest for e.g. studying families of Kähler-Einstein metrics, a topic that has been studied for well over a decade (see \cite{schumacher}, \cite{tsujidynamical}, \cite{cgp1}, \cite{cgp2}, \cite{dinezzafamilies}...), but also in works related to the Kähler-Ricci flow, to collapsing families of Calabi-Yau manifolds... We refer the reader to Section \ref{subsect_misc} for some discussion and new questions in various directions. A main interest of allowing such general families lies, for example, in Donaldson's example (\cite[Section 5.1]{donbstability}) of a special fibre not accessible via test configurations.

\paragraph{The complex side of the story: the space of degenerating metrics.}

We will be working on meromorphic degenerations (or simply degenerations) $\pi:X\to\pdisc$. Those are projective holomorphic submersions onto the punctured disc, with compact Kähler fibres, subject to the additional condition that they are isomorphic to an object of the form $\cX-\cX_0$, where $\cX\to \disc$ is a normal complex analytic space and $\cX_0$ is its fibre over $0$. The fibres in such families are smoothly isomorphic, but they are not in general biholomorphic: one can imagine such families as a degeneration of the complex structures of the underlying smooth manifold. Such a $\cX$ is what we call an analytic model (or simply a model) of $X$. If $\cL$ is a line bundle on $\cX$, isomorphic to $L$ away from the central fibre and compatible with the above identifications, we say that $(\cX,\cL)$ is a model of $(X,L)$.

\bsni We will be interested in spaces of psh metrics on line bundles $L$ over a degeneration $X$. 

\bsni Fixing a fibre $X_z$ for the moment, Darvas introduces in \cite{darv} a $L^1$-type metric structure on the space of finite-energy metrics $\cE^1(X_z,L_z)$, where the distance between two metrics $\phi_{0,z},\phi_{1,z}$ is given by
$$d_{1,z}(\phi_{0,z},\phi_{1,z})=E(\phi_{0,z})-E(\phi_{1,z})+2E(P(\phi_{0,z},\phi_{1,z})).$$
The term $E$ is the Monge-Ampère (or Aubin-Yau) energy of a psh metric, which is a normalized primitive of the Monge-Ampère operator
$$\phi_z\mapsto (dd^c\phi_z)^d$$
for $\phi_z$ smooth, and which is extended to general finite-energy metrics via decreasing limits, as in \cite{bbgz}.
On the other hand, the $P$ term is the "rooftop" envelope
$$P(\phi_{0,z},\phi_{1,z})=\sup\{\phi_z\in\PSH(X_z,L_z),\,\phi_z\leq\phi_{0,z},\phi_{1,z}\},$$
which generalizes the convex envelope of the minimum of two convex functions. If the two metrics in the left-hand side have finite energy, then their rooftop envelope can also be shown to have finite energy, so that all the terms in the above expression are finite.

\bsni We wish to define a distance on the space of metrics on $L$ (with fibrewise finite energy) via taking the Lelong number of the collection of fibrewise distances above. This is to be understood as a generalized (signed) Lelong number as follows: given a subharmonic function $f$ such that there exists a real number $a$ making $f+a\log|z|$ bounded above, we define
$$\hat f = \lim_{r\to 0}\frac{\sup_{|z|=r}f(z)+a\log|z|}{-r}-a.$$
Assuming $z\mapsto d_1(\phi_{0,z},\phi_{1,z})$ to satisfy this claim, we would then define
$$\hat d_1(\phi_0,\phi_1):=\lim_{r\to 0}\frac{\sup_{|z|=r}d_1(\phi_{0,z},\phi_{1,z}+a\log|z|}{-r}-a.$$
As the reader can see, we have made two assumptions: that  the distance between $\phi_{0,z}$ and $\phi_{1,z}$ has at most logarithmic growth, and that this distance function is subharmonic on the disc.

\bsni The first assumption is implied by a natural condition on such metrics $\phi$, which we abusively also call \textit{logarithmic growth}, and can equivalently be described as requiring that $\phi$ extends plurisubharmonically over the central fibre of some analytic model of $X$. This is a very general condition and really the necessary minimum to ask for; therefore, we will denote $\PSH(L)$ the space of psh metrics of logarithmic growth on $L$. In the case of rays, this corresponds to the usual linear growth condition (as in \cite{bbj}).

\bsni Regarding the subharmonicity statement, if one looks at the $\bbs^1$-invariant case, similar statements (see \cite{bdl}) require that rays be \textit{geodesic}. This suggests that in our setting, we must look at maximality conditions coming from the base. In our setting, this is understood in a "relative" pluripotential theory sense, i.e. a metric $\phi$ with logarithmic growth on $L$ is relatively maximal if for any relatively compact open set $U$ on the base $\pdisc$, $\phi|_{\pi\mi(U)}$ is larger than all psh metrics on $L|_U$ bounded above by $\phi$ on $\partial U$. The space of fibrewise finite-energy, relatively maximal, logarithmic growth metrics on $L$ is denoted by $\cE^1(L)$. Note that, although there appears to be many assumptions, this is exactly the generalization of the setting considered in \cite{bbj} and \cite{darlurays}; more precisely, assuming $\bbs^1$-invariance, our space $\hat\cE^1(L)$ corresponds with the space $\cR^1(L)$ of \cite{darlurays}, and our distance $\hat d_1$ corresponds with the distance
$$d_1^\cR(\phi_0,\phi_1):=\lim_{t\to\infty}t\mi d_1(\phi_{0,t},\phi_{1,t}).$$
We then show the following:

\bsni \textbf{Theorem A (\ref{thm_1mainthmsepdiscs}, \ref{thm_1completeness}, \ref{thm_1geodesics}).}
The space $(\hat\cE^1(L),\hat d_1)$ is a complete, geodesic metric space, admitting maximal $\hat d_1$-geodesic segments, constructed as collections of fibrewise geodesic segments.

\bsni The term \textit{maximal segment} in the statement regarding geodesics is to be understood as: the largest segment $t\mapsto\phi_t\in\hat\cE^1(L)$ joining the endpoints which is plurisubharmonic in the variable $t$. In order to avoid confusion with the \textit{relatively maximal metrics} defined above, we will sometimes speak of distinguished segments instead.

\paragraph{Going non-Archimedean.}

As explained for example in \cite{favre} and \cite{berkhodge} (see also \cite{bjtrop}, \cite{yangli}, \cite{shivaprasad}), one can interpret a projective (meromorphic) degeneration $X$ as a projective variety over the field $\K=\bbc((t))$. This is a discretely-valued field, i.e. the possible values of its usual valuation form a discrete additive subgroup of $\bbr$.

\bsni As in the trivially valued case discussed at the very beginning, we define a metric $\phi\an$ on the Berkovich analytification $L_\K\an$ of $L$ seen as a $\K$-line bundle, associated to a metric $\phi$ in $\PSH(L)$, which captures the algebro-geometric information of the singularities of $\phi$ via generic Lelong numbers along vertical divisors in models of $X$.

\bsni We prove that the metric $\phi\na$ is plurisubharmonic in the non-Archimedean sense (see Section \ref{sect_nappt}), which roughly means that it satisfies the statement of Demailly's regularization Theorem, i.e. can be approximated by a decreasing net of Fubini-Study metrics.

\paragraph{Energies and Deligne pairings.}

More structure arises when one considers finite-energy metrics. For this, it will be more convenient to express complex and non-Archimedean Monge-Ampère energies as metrized Deligne pairings (see Section \ref{sect_deligne}). Briefly, one can associate to a flat projective morphism of complex manifolds
$$\pi:X\to Y$$
of relative dimension $d$, and to $d+1$ pairs $(L_i,\phi_i)$ of relatively ample line bundles $L_i$ over $X$ together with continuous psh metrics $\phi_i$, a line bundle
$$\langle L_0,\dots,L_d\rangle_{X/Y}$$
together with a metric
$$\langle \phi_0,\dots,\phi_d\rangle_{X/Y}$$
in a multi-additive and symmetric fashion. Furthermore, this construction satisfies
\begin{itemize}
\item a change of metric formula: given another continuous psh metric $\phi_0'$ on $L_0$,
\begin{equation}
\langle \phi_0,\dots,\phi_d\rangle_{X/Y} - \langle \phi_0',\dots,\phi_d\rangle_{X/Y} = \pi_*\left((\phi_0-\phi_0')(dd^c\phi_1\wedge\dots\wedge dd^c\phi_d)\right);
\end{equation}
\item and a pushforward formula:
\begin{equation}
dd^c\langle \phi_0,\dots,\phi_d\rangle_{X/Y}=\pi_*(dd^c\phi_0\wedge\dots\wedge dd^c\phi_d).
\end{equation}
\end{itemize}
If $Y$ is a point, we omit the subscript $\cdot_{X/Y}$, and upon contemplating the change of metric formula, one can see that the (relative) Monge-Ampère energy between two metrics $\phi_0$ and $\phi_1$ may be written as
$$\langle \phi_0^{d+1}\rangle - \langle \phi_1^{d+1}\rangle,$$
suggesting that the Monge-Ampère energy could be seen as a genuine metric $\langle \phi^{d+1}\rangle$.

\bsni In the present article, we also extend the Deligne pairing construction in the complex setting, from continuous psh metrics to fibrewise finite-energy metrics. We direct the interested reader to Section \ref{sect_deligne}.

\paragraph{Perfecting the complex/non-Archimedean correspondance.}

In the space of non-Archimedean psh metrics, one can also define the subclass of finite-energy metrics $\cE^1(L\an)$, by defining for example the Monge-Ampère energy $E\na$ as a non-Archimedean Deligne pairing (using the machinery of \cite{be20}). (Other approaches, which all give the same functional, include the intersection theory-based formalism of Gubler (\cite{gublerinventionestropical}, and the locally tropical approach of Chambert-Loir and Ducros (\cite{cld}) based on the superforms of Lagerberg (\cite{lagerberg}).) A natural question is now whether the Monge-Ampère energy is equivariant under the "Archimedean to non-Archimedean" map: that is, whether the slope at infinity of the Monge-Ampère energy along a metric $\phi$ in $\hat\cE^1(L)$ coincides with the non-Archimedean Monge-Ampère energy of its associated non-Archimedean metric $\phi\na$, which we can write (in a way that will be made precise in Remark \ref{rem_namap}) as
\begin{equation}\label{eq_ext}E(\phi)\na=E\na(\phi\na)
\end{equation}
(seeing Monge-Ampère energies as Deligne pairings, as previously), where the term on the left can be understood as a generalized Lelong number at zero of the Monge-Ampère energy along $\phi$. It turns out that, in general, we only obtain an inequality
$$E(\phi)\na\leq E\na(\phi\na),$$
as in Corollary \ref{coro_eonemapping}.

\bsni It becomes natural to take our interest to the class of metrics in $\hat\cE^1(L)$ for which equality holds. We define them as "hybrid maximal" metrics, which correspond to the maximal psh geodesic rays of \cite{bbj}. We give equivalent characterizations of such metrics, in particular using more "complex pluripotential"-theoretic notions such as an extremal characterization. Denoting the space of hybrid maximal metrics by $\hat\cE^1_{\mathrm{hyb}}(L)\subset \hat\cE^1(L)$, we then prove that it can be completely realized by the space $\cE^1(L\an)$.

\bsni \textbf{Theorem B (\ref{thm_isometry}, \ref{thm_uniquenesstropmaxdisc} \ref{prop_functionals}).}
\begin{itemize}
\item There is an isometric embedding of $(\cE^1(L\an),d_1\na)$ into $(\hat\cE^1(L),\hat d_1)$ with image $\hat\cE^1_{\mathrm{hyb}}(L)$; this image is characterized as the class of metrics for which \eqref{eq_ext} holds.
\item A psh segment in $\hat\cE^1_{\mathrm{hyb}}(L)$ is a psh geodesic if and only if its image in $\cE^1(L_\K\an)$ is a psh geodesic in the sense of \cite{reb20b}.
\item There is a general "plurifunctional extension" property in this space, as follows: suppose given $d+1$ relatively ample line bundles $L_i$ on $X$. Then, for any $(d+1)$-uple of metrics $\phi_i\in\hat\cE^1_{\mathrm{hyb}}(L_i)$, we have
$$(\langle\phi_0,\dots,\phi_d\rangle_{X/\pdisc})\na=\langle\phi\na_0,\dots,\phi\na_d\rangle.$$

\end{itemize}

\bsni This result states that degenerations of psh metrics that are interesting from the "energy functional" point of view can be studied using purely non-Archimedean techniques; conversely, we may also deduce properties of non-Archimedean metrics from their complex counterparts. A striking example is a consequence of the second statement regarding geodesics: it implies that the non-Archimedean geodesics constructed by the author in previous works (\cite{reb20b}) can be realized as natural limits of complex geodesics! In Section \ref{subsect_misc}, we extract from this a general and very practical principle for proving convexity results for non-Archimedean functionals provided we know convexity of their complex counterparts (Heuristic \ref{heur_convexity}).

\bsni We also derive some new remarks and questions regarding families of Kähler-Einstein metrics in Section \ref{subsect_misc}. Since the aforementioned result concerning non-Archimedean limits of geodesics is also new in the trivially valued case. In Section \ref{subsect_misc}, we explain how our results encompass the trivially-valued case.

\bsni Furthermore, this allows us (Examples \ref{exa_ding} and \ref{exa_mabuchi}) to show that the non-Archimedean Ding energy is convex on $\cE^1(L_\K\an)$. Convexity of the Ding energy appears, in the trivially-valued case for continuous psh metrics, in \cite{blumliu}, where this is used to prove a very general version of the Hamilton-Tian conjecture. We expect that our results will play an important role in such questions in the case of non-equivariant degenerations.

\bsni By similar arguments, we can also show that the non-Archimedean K-energy is convex on $\cE^1(L\an)$, if one assumes the entropy approximation conjecture of \cite{bjkstab}.

\paragraph{Organization of the paper.}

In Section \ref{sect_sect1}, we study spaces of psh metrics with relative finite energy in the general relative case. In Section \ref{sect_sect15}, we specialize to degenerations over the punctured disc, and study our space $\hat\cE^1(L)$, building up to Theorem A. In Section \ref{sect_sect2}, we construct the "Archimedean to non-Archimedean" map, after recalling some notions of non-Archimedean pluripotential theory. giving Theorem C. In Section \ref{sect_sect3}, we study the non-Archimedean limits of energy functionals, then hybrid maximal metrics, proving Theorem B. We also study the trivially-valued case in Section \ref{subsect_36}, connections with convexity of non-Archimedean functionals, and discuss other applications regarding families of Kähler-Einstein metrics.

\paragraph{Acknowledgements.} The author wishes to express his gratitude his advisors, Sébastien Boucksom and Catriona Maclean. He greatly thanks Tamas Darvas for some interesting discussions, and clarifications on the geodesic ray case, especially pertinent to the last Section of this manuscript. He also thanks Robert Berman, Bo Berndtsson, Vincent Guedj, Mattias Jonsson, Steve Zelditch, and Mingchen Xia for various discussions, as well as remarks on a preliminary version of this manuscript.

\newpage

\section{Relative finite-energy spaces.}\label{sect_sect1}

\subsection{Reminders on finite-energy spaces.}

We begin with some reminders concerning $d_1$-structures on spaces of finite-energy metrics in the classical setting. We thus consider a fixed \textit{compact} Kähler manifold $X$, with $\dim X =: d$, endowed with an ample line bundle $L$. 

\bsni Consider two metrics $\phi_0$, $\phi_1\in C^0\cap \PSH(L)$. Their relative Monge-Ampère energy is the quantity
$$E(\phi_0,\phi_1)=\frac{1}{(L^d)(d+1)}\sum_{i=0}^{d}\int_X (\phi_0-\phi_1)\,(dd^c\phi_0)^i\wedge(dd^c\phi_1)^{d-i}.$$
Note that we have a cocycle identity
$$E(\phi_0,\phi_1)=E(\phi_0,\phi')+E(\phi',\phi_1)$$
for any other continuous psh metric $\phi'$. Having fixed a continuous psh metric $\refmetric$ on the right, $E(\phi):=E(\phi,\refmetric)$ can be seen as an operator on $C^0\cap\PSH(L)$ which is also a primitive of the Monge-Ampère operator $\MA:\phi\mapsto (dd^c\phi)^d$. It admits a (possibly infinite) extension to $\PSH(L)$ via
$$E(\phi)=\lim_{k\to\infty}E(\phi_k),$$
where $\phi_k$ is a net of continuous psh metrics decreasing to $\phi$, which always exists by \cite{blockolo}, \cite{demcurrents}. The space of finite-energy metrics is the space
$$\cE^1(L)=\{\phi\in \PSH(L),\,E(\phi)\text{ is finite}\}.$$
By the cocycle identity, this space does not depend on the choice of a reference metric. From the work of Darvas, we know this space to admit (amongst others) a $d_1$-type complete metric space structure via
$$d_1(\phi_0,\phi_1)=E(\phi_0)+E(\phi_1)-2E(P(\phi_0,\phi_1)),$$
where $P(\phi_0,\phi_1)$ is the envelope
$$P(\phi_0,\phi_1)=\sup\,\{\phi\in \PSH(L),\,\phi\leq\min(\phi_0,\phi_1)).$$
It will be more practical to use a different expression of the Monge-Ampère energy, as a difference of absolute Deligne pairings.

\subsection{Relative finite-energy metrics and extended Deligne pairings.}\label{sect_deligne}

The Deligne pairing has a long history, starting from Deligne's original article, treating the case of relative dimension $1$ (\cite{deligne}), further generalized by Elkik in \cite{elkik1}, \cite{elkik2}. Its use to formulate functionals arising complex geometry has been popularized via \cite{phongrosssturm}, and recently, Deligne pairings have also been shown to be of great use in non-Archimedean geometry (\cite{bhjasymptotics}, \cite{be20}, see also \cite[Remark 6]{phongrosssturm}). The non-Archimedean case over a point has been thoroughly developed in \cite{be20}.

\bsni We consider a holomorphic submersion between complex manifolds
$$\pi:X\to Y$$
of relative dimension $d$. Pick $d+1$ pairs $(L_i,\phi_i)$, where $L_i$ is a relatively ample line bundle over $X$, and $\phi_i$ is a continuous psh metric on $L_i$. To this data, one associates a line bundle over $Y$,
$$\langle L_0,\dots,L_d\rangle_{X/Y},$$
together with a metric
$$\langle \phi_0,\dots,\phi_d\rangle_{X/Y}$$
in a way that is multi-additive, symmetric; the construction furthermore commutes with base change (in particular, is stable upon restriction to an open set on the base), and satisfies
\begin{itemize}
\item the change of metric formula: given another continuous psh metric $\phi_0'$ on $L_0$, we have
\begin{equation}\label{eq_changeofmetricformula}
\langle \phi_0,\dots,\phi_d\rangle_{X/Y} - \langle \phi_0',\dots,\phi_d\rangle_{X/Y} = \pi_*\left((\phi_0-\phi_0')(dd^c\phi_1\wedge\dots\wedge dd^c\phi_d)\right)
\end{equation}
(see \cite[Théorème I.1.1(d)]{elkik2});
\item the curvature formula
\begin{equation}\label{eq_ddc}
dd^c\langle \phi_0,\dots,\phi_d\rangle_{X/Y}=\pi_*(dd^c\phi_0\wedge\dots\wedge dd^c\phi_d)
\end{equation}
(see \cite[Théorème I.1.1(d)]{elkik2}).
\end{itemize}
The last formula shows that the metric $\langle \phi_0,\dots,\phi_d\rangle_{X/Y}$ is positive. One also notices that a metric on a trivial Deligne pairing
$$\langle \cO_X,L_1,\dots,L_d\rangle_{X/Y}$$
can be identified with a genuine function on the base $Y$, upon evaluation against the trivial section $1$.

\bsni Assume for the moment that $Y$ is a point. In that case, Deligne pairings can be seen as complex lines together with a Hermitian norm. In this setting, we will omit the subscript $\cdot_{X/Y}$. Using the change of metric formula, one can see the relative Monge-Ampère energy between two continuous psh metrics on a fixed line bundle $L$ over $X$ as a difference of Deligne pairings:
$$(d+1)E(\phi_0,\phi_1)=\langle \phi_0^{d+1}\rangle-\langle \phi_1^{d+1}\rangle,$$
which suggests that the Monge-Ampère energy should be seen intrinsically as a genuine (Hermitian) metric
$$(d+1)E(\phi)=\langle \phi^{d+1}\rangle$$
on the line $\langle L^{d+1}\rangle$.

\bsni We now return to arbitrary $Y$. The change of metric formula suggests that the Deligne pairing construction could possibly make sense in a larger class of metrics, where each $\phi_i$ has fibrewise finite energy. This motivates the following definition.
\begin{defi}Let $L$ be a relatively ample line bundle on $X$. We define the class of relative finite-energy metrics $$\cE^1_{X/Y}(L)$$ to be the class of plurisubharmonic metrics $\phi$ on $L$ such that, for all $y\in Y$, $\phi_y\in\cE^1(L_y)$. Here, $L_y$ is the restriction of $L$ to the fibre $\pi\mi(y)$.
\end{defi}

\noindent Since we have required plurisubharmonicity on all of $L$, it follows that any metric in $\cE^1_{X/Y}(L)$ can be approximated by a decreasing net of continuous psh metrics on $L$. In particular, such metrics admit Deligne pairings.

\begin{theorem}\label{thm_eonepairing}Let $\pi:X\to Y$ be a holomorphic submersion between complex manifolds of relative dimension $d$, and let $(L_i)_{i=0}^d$ be a collection of $d+1$ relatively ample line bundles on $X$. There exists a unique extension of the Deligne pairing construction to metrics in $\cE^1(L_i)_{X/Y}$, which is multilinear, symmetric, stable upon restriction to a smaller open set on the base, and such that the change of metric formula \eqref{eq_changeofmetricformula} holds. 

\end{theorem}
\begin{proof}
We first restrict to an open set $U$ on the base $Y$, so that we may apply Demailly regularization on $\pi\mi(U)$. Fix for each $i$ a metric $\phi_i\in\cE^1_{X/U}(L_i)$, and let $k\mapsto \phi_i^k$ be a sequence of continuous psh metrics on $L_i$ decreasing to $\phi_i$. We claim that the sequence
$$k\mapsto \langle \phi_0^k,\dots,\phi_d^k\rangle_{X/U}$$
decreases to a finite-valued metric on $U$, independent of the choices of approximating sequences, which defines our construction restricted to $U$. Assuming this convergence to hold, one sees that this construction is multilinear, symmetric, satisfies the change of metric formula. Uniqueness follows from the change of metric formula, which itself shows that the construction glues well over $X$.

\bsni That it would define a finite-valued metric on $U$ follows from Lemma \ref{lem_finitepairings} below, so that all that is left in order to prove the Theorem is that the limit in question is decreasing. We proceed by induction on the number $n$ of indices $i\in\{0,\dots,d+1\}$ such that $\phi_i$ belongs to $\cE^1_{X/U}(L_i)-C^0\cap\PSH(L_i)$. In the case $n=0$, all metrics are continuous psh and this is the classical Deligne pairing, so that we have nothing more to prove.

\bsni Assume thus that the assertion holds for some $d+1>n>0$. Assume the metrics $\phi_i$, $i=1,\dots,d+1-n$ to belong to $C^0\cap\PSH(L_i)$, and the $n+1$ other metrics $\phi_i$ to belong strictly to $\cE^1_{X/U}(L_i)$, $i=0$ or $i=d+2-n,\dots,d$. (We can do this without loss of generality, by symmetry and up to reordering the indices.) We approximate $\phi_0$ and the $(\phi_i)_{i=d+2-n}^{d}$ by sequences $k\mapsto \phi_i^k$ of continuous psh metrics. For a fixed $\ell\in\mathbb{N}^*$ and by the induction assumption, the sequence
$$k\mapsto \langle \phi_0^\ell,\phi_1,\dots,\phi_{d+1-n},\phi^k_{d+2-n},\dots,\phi^k_{d}\rangle_{X/Y}$$
is decreasing and converges to a limit
$\langle \phi_0^\ell,\phi_1,\dots,\phi_{d+1-n},\phi_{d+2-n},\dots,\phi_{d}\rangle_{X/U}$. This limit satisfies, for any fixed metric $\phi'_0\in C^0\cap\PSH(L_0)$ the formula
\begin{align*}
&\langle \phi_0^\ell,\phi_1,\dots,\phi_{d+1-n},\phi_{d+2-n},\dots,\phi_{d}\rangle_{X/U}-\langle \phi'_0,\phi_1,\dots,\phi_{d+1-n},\phi_{d+2-n},\dots,\phi_{d}\rangle_{X/U}\\
&=\int_{X/U}(\phi_0^\ell-\phi'_0)\,dd^c\phi_1\wedge\dots\wedge dd^c\phi_{d+1-n}\wedge dd^c\phi_{d+2-n}\wedge\dots\wedge dd^c\phi_{d}.
\end{align*}
Now, this expression yields a decreasing net as $\ell$ increases, and its limit is finite. In particular, it be seen to be the decreasing limit of
$$k\mapsto \langle \phi_0^\ell,\phi_1,\dots,\phi_{d+1-n},\phi^k_{d+2-n},\dots,\phi^k_{d}\rangle_{X/U},$$
which proves our desired statement by induction.
\end{proof}

\noindent Along the way, we have used the following Lemma regarding finiteness of products of absolute finite-energy classes. The proof follows from exactly the arguments of \cite[Theorem 5.8]{bjtrivval21}, therefore we leave the details to the interested reader.
\begin{lemma}\label{lem_finitepairings}Let $X$ be a compact Kähler manifold of dimension $d$, and let $(L_i)$ be a collection of $d+1$ ample line bundles on $X$. Fix, for all $i=0,\dots,d$, a metric $\phi_i\in\cE^1(L_i)$, and a continuous metric $\phi'_0\in\cE^1(L_0)$. Then, the integral
$$\int_X (\phi_0-\phi'_0)\,dd^c\phi_1\wedge\dots\wedge dd^c\phi_d$$
is finite.
\end{lemma}

\subsection{Relative plurisubharmonic segments.}

In the case where $Y$ is a point (the absolute setting), it is well-known (\cite{darv}) that any two metrics $\phi_0$, $\phi_1$ in $\cE^1_{X/Y}(L)=\cE^1(X,L)$ can be joined by a  plurisubharmonic \textit{geodesic} segment in $\cE^1(X,L)$, in the following sense. There exists a $\bbs^1$-invariant plurisubharmonic metric $\Phi$ on the product $L\times \mathcal{A}$ (where $\mathcal{A}$ is the annulus $\{e\mi\leq |z| \leq 1\}$) identified via $t=-\log|z|$ with a segment
$$[0,1]\ni t\mapsto \Phi_t\in\cE^1(X,L),$$
such that $\Phi$ bounds by above all other such segments $\Psi$ with $\Psi_0\leq\phi_0$ and $\Psi_1\leq\phi_1$. We now look at what happens when $Y$ is no longer a point. Given $\phi_0$, $\phi_1\in\cE^1_{X/Y}(L)$, and a point $y\in Y$ on the base, there exists by the previous discussion a plurisubharmonic geodesic segment $t\mapsto \phi_{t,y}$ joining $\phi_{0,y}$ and $\phi_{1,y}$ in $\cE^1(X_y,L_y)$. Varying $y$, this gives a collection of plurisubharmonic geodesic segments $t \mapsto \phi_t$. It is not obvious that, for given $t$, $\phi_t$ has plurisubharmonic variation with respect to $Y$. We thus claim the following:
\begin{theorem}\label{thm_relsegments}
Given any $\phi_0$, $\phi_1\in\cE^1_{X/Y}(L)$, the collection $[0,1]\ni t \mapsto \phi_t$ of fibrewise psh geodesic segments belongs to $\cE^1_{X/Y}(L)$. Furthermore, identifying the collection $t\mapsto \phi_t$ with a $\bbs^1$-invariant metric $\Phi$ on $L\times \mathcal{A}$, $\Phi$ is plurisubharmonic, and is the unique psh metric on $L\times \bbd^*$ such that for all $y\in Y$,
$$dd^c_t(\pi_y)_*\langle \Phi^{d+1}\rangle_{X\times\mathcal{A}/Y\times\mathcal{A}}=0,$$
where $\pi_y: X\times\mathcal{A}\to \{y\}\times\mathcal{A}$ is the projection to the point $y$. Such a segment will be called a psh geodesic segment in $\cE^1_{X/Y}(L)$.
\end{theorem}
\noindent The last statement can be interpreted as saying that the Monge-Ampère energy along the psh geodesic segment is fibrewise affine.

\bsni This Theorem is proven via families of Bergman kernels. We recall some facts on this topic. For all $z$ on the base, picking a smooth, strictly psh metric $\phi$ on $L$ endows the $H^0(kL_z+K_{X_z})$ (for all positive integers $k$) with a Hermitian norm
$$\|s_z\|_{\phi,z,k}^2=\int_X |s_z|^2 e^{-k\phi_{z}},$$
for $|s_z|^2 e^{-k\phi_{z}}$ is indeed a measure on $X_z$. We may now pick a basis $(s_{j,z})$ which is orthonormal for $\norm_{\phi,z,k}$, and define the Bergman kernel
$$B_{\phi,z,k}:=\sum_j |s_{j,z}|^2.$$
There are two key points regarding this object. The first one, which is easier to see, is that $B_{\phi,z,k}$ is independent of the choice of such a basis. The second, much deeper point is that as we move \textit{on the base}, the Fubini-Study metrics
$$\phi_k:z\mapsto FS_k(\norm_{\phi,z,k}):=k\mi\log B_{t,z,k}$$
vary plurisubharmonically, i.e. define a globally psh metric on $L$. This is a particular case of \cite[Theorem 0.1]{berndtpaun1}. We will return to this construction shortly.

\bsni We now turn to some facts regarding spaces of norms. Given any two Hermitian norms $\norm_0$, $\norm_1$ acting on a complex finite-dimensional vector space $V$, it is a well-known fact that one may find a basis $(s_j)_j$ of $V$ which is orthonormal for $\norm_0$ and also ortho\textit{gonal} for $\norm_1$. One may then define a distinguished segment $t\mapsto \norm_t$ of Hermitian norms by requiring $\norm_t$ to be the unique norm orthogonal in the basis $(s_j)_j$, and such that for all $j$,
$$\|s_j\|_t=\|s_j\|_1^t.$$
Such segments are in fact geodesic for various $d_p$-type metric structures on the space of Hermitian norms on $V$, but we will not need that fact. 

\bsni We can consider as before the psh geodesic $t\mapsto \phi_{t,z}$ joining $\phi_{0,z}$ and $\phi_{1,z}$ in $\cE^1(L_z)$. Then, a result of Berndtsson (\cite[Theorem 1.2]{berndtprob}) states that the Fubini-Study metrics from before approximate the geodesic $\phi_t$ uniformly in $t$: there exists a constant $c=c(z)$ such that
\begin{equation}\label{eq_1berndtprob}
|FS_k(\norm_{t,k,z})-\phi_{t,z}|\leq c(z)\cdot k\mi \log k.
\end{equation}
One would like to have such an approximation to be ($Y$-locally) independent of the variable on the base. Firstly, on reading the proof of \cite[Theorem 1.2]{berndtprob}, one notices that the constant $c$ depends only on the endpoints $\phi_{0,z}$ and $\phi_{1,z}$, so that our problem reduces to knowing whether one can find $c$ such that, for all $z$ in some compact $U$ in $Y$,
$$|FS_k(\norm_{0,k,z})-\phi_{0,z}|\leq c\cdot k\mi \log k$$
(and similarly for $t=1$; the proof is symmetric.) This follows from adapting the general uniform Bergman kernel asymptotics result \cite[Theorem 4.1.1]{mamarinescu} to the case of varying complex structure, along the same lines as explained in the proof of \cite[Theorem 1.6]{mazhang}. The proof of Theorem \ref{thm_relsegments} is now a matter of adequately piecing together all the previous results.

\begin{proof}[Proof of Theorem \ref{thm_relsegments}.]
Let $\phi_0$, $\phi_1\in\cE^1_{X/Y}(L)$, which we assume to be smooth and strictly psh (while the general case follows from regularization), and consider the families of fibrewise psh geodesics $t\mapsto \phi_t$. As explained before, for any compact $U$ on the base, and all $z\in U$, there exists a $c$ independent of $z$ such that
$$|FS_k(\norm_{t,k,z})-\phi_{t,z}|\leq c\cdot k\mi \log k,$$
by \cite[Theorem 1.2]{berndtprob} and the uniform Bergman kernel asymptotics. Furthermore, as also discussed, the families
$$(t,z)\mapsto FS_k(\norm_{t,k,z})$$
have plurisubharmonic variation in $z$ and $t$. Combined with the above estimates, this means that, over $\pi\mi(U)\times\mathcal{A}$ (where $\pi:X\to Y$ is the structure morphism of $X$ over $Y$, and $\mathcal{A}$ is the annulus corresponding to $[0,1]$), the segment $t\mapsto \phi_t$ seen as a metric on $L\times\cA$ can be uniformly approximated by continuous psh metrics on $L\times \cA$, furthermore $\bbs^1$-invariant under the second variable. This settles the first statement of the Theorem.

\bsni That $\phi_t$ would be the unique segment such that
$$dd^c_t(\pi_y)_*\langle \Phi^{d+1}\rangle_{X\times\mathcal{A}/Y\times\mathcal{A}}=0$$
for all $y\in Y$ follows by definition of the fibrewise psh geodesic segments (they are themselves characterized as the unique segments in each $\cE^1(X_y,L_y)$ along which the Monge-Ampère energy is affine).
\end{proof}

\subsection{Relatively maximal metrics.}\label{sect_relmax}

\begin{defi}
Let $\pi:X\to Y$ be a holomorphic submersion with compact Kähler fibres. Let $L$ be a relatively ample line bundle on $X$. We say that a metric $\phi$ on $L$ is relatively maximal if it is maximal in the usual sense of Sadullaev (e.g. \cite{klimek}) on the preimage of any relatively compact open subset of $Y$. In other words, $\phi$ is relatively maximal if and only if, for any relatively compact open subset $U$ of $Y$, for any relatively compact open subset $V$ of $\pi\mi(U)$, and for any psh metric $\psi$ on the restriction of $L$ to $\pi\mi(U)$ such that $$\limsup \psi(z)-\phi(z)\leq 0$$ as $z$ approaches the boundary of $\pi\mi(U)$, then
$$\psi(z)\leq\phi(z)$$
for all $z$ in $\pi\mi(U)$.
\end{defi}

\begin{remark}One sees from this definition that a decreasing limit of relatively maximal psh metric is also relatively maximal.
\end{remark}

\begin{remark}
Let $M$ be a compact Kähler manifold together with an ample line bundle $L_M$. Let $[0,\infty)\ni t \mapsto \phi_t$ be a psh ray of psh metrics on $L_M$. Seen as a $\bbs^1$-invariant psh metric on the product $L\times\pdisc$, $\phi$ is relatively maximal in our sense if and only if it is "geodesic" in the sense of \cite{bbj}.
\end{remark}

\noindent A nice way to generate relatively maximal metrics is via Perron-Bremmermann envelopes, as we now prove. We extend our setting slightly, to allow for singular fibres, which will be useful later on. We state our result in maximal generality, but the case to keep in mind is that of a holomorphic submersion over the punctured disc with a singular fibre over zero.
\begin{theorem}\label{thm_existencemaximal}
Let $\pi:X\to Y$ be a holomorphic projective surjective morphism. Let $\Omega$ be a relatively compact, smooth open subset of $Y$, such that $\pi$ is a submersion above (hence near) $\partial\Omega$. Let $L$ be a $\pi$-ample line bundle on $X$. Let $\phi$ be a continuous collection of fibrewise psh metrics on $\pi\mi(\partial\Omega)$. We then have that:
\begin{enumerate}
\item if there exists a continuous psh extension of $\phi$ to all of $\pi\mi(\Omega)$, then there exists a (unique) relatively maximal continuous psh extension of $\phi$ to all of $\pi\mi(\Omega)$;
\item if $\Omega$ is defined as $\{\rho < 0\}$, where $\rho$ is a smooth strictly psh function on $Y$, such that $\nabla \rho \neq 0$ whenever $\rho=0$, then a continuous psh extension as above exists.
\end{enumerate}
\end{theorem}
\begin{remark}An open subset that satisfies the second point above is sometimes called a \textit{hyperconvex} open subset. In particular, $\bbd$ and annuli centered at zero are such open sets. The proof of the first point follows some ideas dating back to the work of Bedford-Taylor (\cite{btdirichlet}), see e.g. \cite[Proposition 6.3]{bbgz}, \cite{ps06}.
\end{remark}
\begin{proof}[Proof of Theorem \ref{thm_existencemaximal}.] The hypotheses in the Theorem give that $\pi\mi(\bar \Omega)$ is a manifold with boundary, which we denote $M:=\pi\mi(\bar \Omega)$, and whose boundary is $\pi\mi(\partial \Omega)$, which we denote $\partial M:=\pi\mi(\bar \Omega)$. We will finally write $\bar M:=\pi\mi(\bar\Omega)$.

\bsni\textit{1. Existence of a continuous relatively maximal metric, assuming existence of a subsolution.} \\
\noindent We claim that the envelope
$$\cP \phi=\sup{}^*\{\psi\in C^0\cap \PSH(L|_M),\,\psi\leq\phi\text{ on }L|_{\partial M}\}$$
is our desired relatively maximal, continuous metric on $L|_M$ which coincides with $\phi$ on $L|_{\partial M}$. By definition, $\cP \phi$ is relatively maximal; furthermore, since there exists a continuous subsolution, i.e. a candidate $\psi$ to the envelope which coincides with $\phi$ on $L|_{\partial M}$, $\cP\phi$ also has the correct boundary values. We are therefore left to show continuity.

\bsni We begin with a continuity estimate near the boundary. Having fixed a reference smooth, strictly psh metric $\refmetric$ on $L$, and setting $\omega=dd^c\refmetric$, we can see any candidate $\psi$ for the envelope $\cP\phi$ as a continuous $\omega$-psh function $g=\psi-\refmetric$. Fix such a $g$, and set
$$f_0=\phi-(\refmetric)|_{\partial M}.$$
Since $dd^ g\geq -\omega$, the Laplacian $\Delta_\omega g$ of $g$ with respect to $\omega$ is bounded below by $-d-1$. Let $f$ be the (continuous) solution on $\bar M$ to the Dirichlet problem
$$\Delta_\omega f +(d+1) = 0,\,f|_{\partial M}=f_0.$$
We then have that $\Delta_\omega(g-f)\geq 0$, which implies by the maximum principle that
$$\sup_{M} \,(g-f) = \sup_{\partial M}\,(g-f),$$
while this supremum is nonpositive since $\psi=g+\refmetric$ is a candidate for the envelope $\cP\phi$. We then have that
$$g\leq f$$
on all of $\bar M$, and this is true for any candidate $\psi$, so that
$$\cP\phi\leq \refmetric + f$$
on $\bar X$.

\bsni We now look at continuity on $M$. Let $\tilde\phi$ denote a continuous psh extension of $\phi$ to $L|_M$. We fix a very small $\varepsilon>0$, and define
$$U=\{\cP\phi^*<\tilde\phi+\varepsilon\},$$
which is the complement of a small compact set containing $\partial \Omega$. By regularization (e.g. \cite[Theorem 3.8]{boul2}), we can find a sequence $\psi_k\in C^0\cap \PSH(L|_{U})$ which decreases to $\cP\phi$. Now, by Dini, using compactness, we have that $U$ is covered by finitely many of the
$$U_k=\{\psi_k<\tilde\phi+\varepsilon\}$$
(since such inequality holds, for all $z\in M$, and for all large enough $k_z$). In particular, for large enough $k$, one has $\psi_k<\tilde\phi+\varepsilon$. We now define $\tilde\psi_k:=\max(\psi_k-\varepsilon,\tilde\phi)$, which is defined on all of $M$. For all $k$, $\tilde\psi_k$ is continuous, as $\psi_k$ is continuous away from the boundary and $\tilde\phi$ is continuous everywhere (in particular, near and up to the boundary). Furthermore, $\tilde\psi_k$ is equal to $\phi$ on the boundary, so that
$$\psi_k-\varepsilon\leq\tilde\psi_k\leq \cP\phi \leq \cP\phi^* \leq \psi_k.$$
This implies that $\psi_k$ converges uniformly to $\cP\phi$, i.e. $\cP\phi$ is continuous on $M$. Furthermore, since:
\begin{enumerate}
\item $\cP\phi\leq \refmetric + f$, as at the end of the first point of the proof;
\item $\refmetric + f$ converges continuously to $\phi$ near the boundary, and $\cP\phi$ is continuous on $M$;
\item there exists a psh extension $\tilde\phi$, ensuring that $\cP\phi= \phi$ on the boundary,
\end{enumerate}
then $\cP\phi$ is continuous up to the boundary.

\bsni \textit{2. Construction of a subsolution under the hyperconvexity assumption.}\\
The second point of the Theorem will follow from a more general principle: consider the class $\cC(L|_{\partial M})$ consisting of continuous, fibrewise psh metrics on $L|_{\partial M}$ admitting a continuous psh extension (hence a relatively maximal continuous extension, by the first point) to all of $L|_M$. Then, this class is stable under uniform limits, which follows from seeing that the mapping $\cC(L|_{\partial M})\ni \phi\mapsto \cP\phi$ from the first part of the proof is continuous under uniform convergence.

\bsni To prove the second point, we therefore have to show that there exists a sequence $\phi_k\in\cC(L|_{\partial M})$ converging uniformly to our boundary data $\phi$. We proceed bu Bergman kernel approximation. Since $L$ is $\pi$-ample, the sheaves $\pi_*(kL)$ are locally free for all $k$ large enough, and correspond to the sections of a vector bundle $E_k$ whose fibres are the $H^0(kL_z=$, $z\in M$. The collection of $L^2$-norms $N_k(\phi)$ associated to $k\phi$ then define a continuous collection of Hermitian metrics $h_k$ on $E_k|_{\partial \Omega}$. We pick a sequence of smooth families of Hermitian metrics $(h_{k,j})_j$ on $\pi_*(kL)$ so that $h_{k,j}\to h_k$ uniformly on $\pi_*(kL)|_{\partial \Omega}$. The associated collection of metrics
$$\phi_{k,j,z}=FS_k(h_{k,j,z})$$
vary smoothly on $L$. Since they are fibrewise \textit{smooth} and \textit{strictly psh} (both of which are necessary conditions for the following argument), we may compensate for the lack of plurisubharmonicity in the direction of $z$, by pulling back a high enough multiple $m_{k,j}\pi^*\rho$ of the defining function $\rho$ of $\Omega$, which as we recall vanishes on the boundary of $\Omega$. We therefore have a continuous psh extension
$$\phi_{k,j}+m_{k,j}\pi^*\rho$$
of $(\phi_{k,j})|_{\partial M}$. This implies that $\phi_{k,j}\in \cC(L|_{\partial M})$; furthermore, $(\phi_{k,j})|_{\partial M}\to \phi_k$ uniformly, which implies that $\phi_k\in\cC(L|_{\partial M})$. Now, by Bergman kernel approximation, the $\phi_k$ themselves converge uniformly \textit{and increasingly} to $\phi$, which implies $\phi\in\cC(L|_{\partial M})$, concluding the proof.
\end{proof}
\begin{remark}By adapting classical arguments of pluripotential theory, one shows that a continuous psh metric on $L$ is relatively maximal iff $(dd^c\phi)^{d+1} = 0$ on $X$, iff it coincides over each relative open subset with its Perron-Bremmermann envelope as in the above Theorem.\end{remark}

\noindent We now characterize relatively maximal metrics of relative finite energy.
\begin{prop}\label{prop_harmonicitymaximaldiscs}
Let $\phi$ be a metric in $\cE^1_{X/Y}(L)$. Assume that 
$Y$ is covered by relatively compact hyperconvex smooth open subsets. Then, $\phi$ is relatively maximal if and only if $\langle \phi^{d+1}\rangle_{X/Y}$ has zero curvature.
\end{prop}
\begin{proof}
Assume $\phi$ to be relatively maximal. Since the Deligne pairing construction is stable upon restriction to an open set, we will work over the preimage $U$ of some smooth hyperconvex relatively compact open set in $Y$. If $\phi$ is continuous, we have just mentioned that $(dd^c\phi)^{d+1} = 0$ there, so that  by \eqref{eq_ddc}, it follows that $dd^c E(\phi)\equiv0$. The non-continuous case follows from regularization on $U$: pick a sequence of continuous metrics $\phi_k$ decreasing to $\phi$ on $U$; by Theorem \ref{thm_existencemaximal}, there exists a continuous, relatively maximal psh metric $\Phi_k$ coinciding with $\phi_k$ on $L|_{\pi\mi(\partial U)}$. By maximality, the sequence $\Phi_k$ necessarily converges to $\phi$ (since $\phi$ is assumed to be relatively maximal), and continuity of Deligne pairings along decreasing nets ensures $\langle \phi^{d+1}\rangle_{X/Y}$ to have zero curvature.

\bsni Conversely, assume $\langle \phi^{d+1}\rangle_{X/Y}$ to have zero curvature. In the continuous case, using \eqref{eq_ddc} again, it follows that $(dd^c\phi)^{d+1}=0$, as it is a nonnegative measure. In the general case, we again proceed base-locally, and approximate $\phi$ on the preimage of a relatively compact open subset $U$ via a decreasing sequence of continuous psh metrics $k\mapsto \phi_k$. Let $\Phi_k$ be for each $k$ the unique continuous and relatively maximal metric on $U$ with prescribed boundary condition $\phi_k|_{\pi\mi(\partial U)}$, given by Theorem \ref{thm_existencemaximal}. Let $\Phi$ denote the limit of the decreasing sequence $k\mapsto \Phi_k$, which is relatively maximal. By continuity of the Deligne pairing along decreasing nets, this sequence also defines a decreasing sequence of zero curvature metrics $\langle \Phi_k^{d+1}\rangle_{X/U}$
which has to converge to the metric
$$\langle \Phi^{d+1}\rangle_{X/U},$$
which is a zero curvature metric $\tilde\phi$ on $U$, coinciding on $\partial U$ with $\langle \phi^{d+1}\rangle_{X/U}$. Since $\langle \phi^{d+1}\rangle_{X/U}$ also has zero curvature, we have to have
$$\langle \Phi^{d+1}\rangle_{X/U}=\langle \phi^{d+1}\rangle_{X/U}$$
on all of $U$. Fix $z$ in $U$, and note that this implies
$$E(\Phi_z)=E(\phi_z),$$
while by relative maximality of $\Phi$, $\phi_z\leq \Phi_z$, which implies $\Phi_z=\phi_z$, thus concluding our proof.
\end{proof}

\section{Finite-energy metrics over degenerations.}\label{sect_sect15}

\subsection{Analytic models and degenerations.}

We now turn to our main setting. We will consider the base $Y$ to be the punctured unit disc, and we will assume that our family degenerates (meromorphically) as one approaches zero.

\begin{defi}
Consider a holomorphic submersion $\pi:X\to\pdisc$ with compact Kähler fibres, and a relatively ample line bundle $L$ on $X$. An analytic model (or simply a model) of $X$ is a normal complex analytic space $\cX$, together with a flat, proper holomorphic morphism $\pi:\cX\to\overline\bbd$, realizing an isomorphism $X\simeq \pi\mi(\overline\bbd^*)$. An analytic model of $(X,L)$ is the data of an analytic model $\cX$ on $X$, and an ample line bundle $\cL$ over $\cX$ such that $\cL$ restricted to $\pi\mi(\overline\bbd^*)$ is isomorphic to $L$. We define a degeneration (or a degeneration with meromorphic singularities) to be a morphism $\pi:X\to\pdisc$ as above, such that there exists an analytic model of $(X,L)$.
\end{defi}

\begin{example}
This construction specializes to the following well-known cases:
\begin{itemize}
\item if all the fibres of $X$ are isomorphic to $M$, a model $\cX$ can simply be viewed as a compactification of an isotrivial degeneration of $M$;
\item if the above condition holds, and furthermore the isomorphism is generated by a $\bbc^*$-action, this is simply a (real) one-parameter degeneration of $(M,L|_M)$, i.e. a test configuration for $(M,L|_M)$.
\end{itemize}
\end{example}

\noindent The central fibre of a model of $X$ is the space $\cX_0=\pi\mi(\{0\})$. If the degeneration $X\to\overline\bbd^*$ is isotrivial, we say that $M$, the fibre over $1$, is the generic fibre of $X$.

\subsection{Generalized slopes and Lelong numbers.}\label{sect_lelong}

As we will be working with (generalized) subharmonic functions on the base $\pdisc$, we will often have to work with some notions of Lelong numbers. We review some (old and new) facts in this Section.

\begin{defi}We say that a subharmonic function $f$ on $\bbd^*$ has logarithmic growth (near zero) if there exists a real number $a$ such that $f(z)+a\log|z|$ is bounded above near zero.
\end{defi}

\noindent In particular, a subharmonic function $f$ with logarithmic growth can be extended as a subharmonic function over the entire disc. In this case, one can define its generalized (in the sense that it is possibly signed) Lelong number, as follows. Pick a number $a$ as above. The function
$$g(t):=\sup_{|z|=e^{-t}} (f(z)+a\log|z|),$$ for $t\in [0,\infty)$, is then a convex function of $t$, and the rate of change
$$\frac{g(t)-g(0)}{t}$$
is thus an increasing nonnegative function of $t$, which has a finite value as $t\to\infty$. This corresponds to saying that the limit
$$\lim_{r\to 0}\frac{\sup_{|z|=r} f(z)+a\log|z|}{-\log r}$$
exists and is finite.

\begin{defi}Given a subharmonic function $f$ with logarithmic growth on $\pdisc$, we define its generalized slope (or generalized Lelong number at zero) to be the value
$$\hat f:=\lim_{r\to 0}\frac{\sup_{|z|=r} f(z)+a\log|z|}{-\log r}+a,$$
where $a$ is a real such that $f+a\log|z|$ is bounded near zero. In particular, $\hat f$ is independent of the choice of such an $a$.
\end{defi}

\begin{example}In the case of an $\bbs^1$-invariant subharmonic function $f$, i.e. a convex function on $[0,\infty)$, this simply computes the slope at infinity
$$\lim_{t\to\infty}\frac{f(t)}{t}.$$
\end{example}

\begin{remark}As a consequence of Harnack's inequality, $\hat f$ may equivalently be computed using the integrals
$$\fint_{|z|=r}f(z)\,dz$$
in place of the suprema.
\end{remark}

\noindent The following estimate will be very useful later on.
\begin{lemma}\label{lem_lelongestimate}
Let $f$ a subharmonic function $f$ with logarithmic growth on $\pdisc$. Then, for all $z$, we have
$$f(z)\leq \log(1/|z|)\cdot \hat f+c,$$
where $c=\sup_{|z|=1} f(z)$.
\end{lemma}
\begin{proof}
Define the function
$$g(t):=\sup_{|z|=e^{-t}} (f(z)+a\log|z|)$$
as before. Since the rate of change
$$\frac{g(t)-g(0)}{t}$$
is an increasing function of $t$, we have for all $s\in[0,\infty)$, all $|z|=e^{-s}$ and since, having fixed our boundary data, $v(0)=0$, we have
$$\frac{f(z)+a\cdot \log|z| - g(0)}{-\log|z|}\leq \frac{g(s)-g(0)}{s}\leq \lim_{t\to\infty}\frac{g(t)}{t}=\hat f - a.$$
Adding $a$ then concludes the proof.
\end{proof}

\subsection{Plurisubharmonic metrics on degenerations.}

Fix now a degeneration $\pi:X\to \overline\bbd^*$, endowed with a relatively ample line bundle $L$. We will take our interest to plurisubharmonic metrics on $L$, and in particular their singularities. However, a general psh metric on a degeneration can behave very poorly near the singularity, even though we have assumed existence of an analytic model of $X$. Thus, we need to enforce a rather natural growth condition on such psh metrics, akin to that of linear growth for geodesic rays.


\begin{defi}\label{def_tame}
We say that a psh metric $\phi$ on $L$ has \textit{logarithmic growth} if there exists a model $(\cX,\cL)$ of $(X,L)$ such that $\phi$ extends as a psh metric on $\cL$.
\end{defi}

\noindent We will write $\PSH(L)$ for the space of psh metrics of logarithmic growth on $L$. If it comes to be necessary, we will rather write $\PSH(X,L)$ when considering the space of (non-necessarily of logarithmic growth) psh metrics on $L$. We will soon show that $\PSH(L)$ has many desirable properties. We will also shortly explain our terminology. We begin with the following result:
\begin{lemma}\label{lem_extpshtame}Given a psh metric $\phi$ on $L$, the following are equivalent:
\begin{enumerate}
\item [(i)] $\phi$ has logarithmic growth, i.e. there exists a model $(\cX,\cL)$ such that $\phi$ extends to a psh metric on $\cL$;
\item [(ii)] for all models $(\cX,\cL)$ of $(X,L)$, there exists a constant $c=c(\cX,\cL)$ such that $\phi+c\cdot \log|z|$ extends to a psh metric on $\cL$;
\item[(iii)] there exists a model $(\cX,\cL)$ and a smooth metric $\refmetric$ on $\cL$ such that 
$$\rho^*\phi(z) \leq \refmetric(z) + O(\log |z|)$$
as $z\to 0$, where $\rho$ denotes the isomorphism between $X$ and $\cX-\cX_0$;
\item[(iv)] for all models $(\cX,\cL)$ of $(X,L)$ and all smooth metrics $\refmetric$ on $\cL
$, (iv) holds.
\end{enumerate}
\end{lemma}
\begin{proof}
By classical results of pluripotential theory, (i)$\Leftrightarrow$(iii) and (ii)$\Leftrightarrow$(iv). Since (iv)$\Rightarrow$(iii) is immediate, we only need to prove (iii)$\Rightarrow$(iv). Assume that 
$$\rho^*\phi(z) \leq \phi_\cL(z) + O(\log |z|)$$
for a smooth reference metric $\phi_\cL$ on $\cL$. Pick another model $(\cY,\cM)$ together with a smooth metric $\phi_\cM$. Note that the equation above holds if and only if the same equation holds for the pullbacks of $\phi_\cL$ and $\rho^*\phi$ to a higher model. Thus, we pick a model $(\cZ,\cN)$ dominating both via $\pi_\cX:\cZ\to \cX$, $\pi_\cY:\cZ\to\cY$. There exists a unique Cartier divisor $D$ supported on the special fibre $\cZ_0$ such that
$$\pi_\cX^*\cL+D=\pi_\cY^*\cM,$$
and given a local equation $f_D$ for $D$, we have
$$\pi_\cX^*\phi_\cL \leq \pi_\cY^*\phi_\cM - \log|f_D| + O(1)\leq \pi_\cY^*\phi_\cM + O(\log|z|).$$
Thus,
$$\pi_\cX^*\rho^*\phi\leq \pi_\cX^*\phi_\cL + O(\log|z|)\leq \pi_\cY^*\phi_\cM + O(\log|z|),$$
as desired.
\end{proof}

\begin{remark}The above result shows that one could equivalently define our growth condition using some fixed reference data $(\cX_\mathrm{ref},\cL_{\mathrm{ref}})$, using e.g. point (ii). In the isotrivial case, there furthermore exists some very natural reference data: the "trivial model" given by the product family of the generic fibre with the whole disc.
\end{remark}

\begin{example}
Let $[0,\infty)\ni t\mapsto \phi_t$ be a ray of psh metrics on an ample line bundle $L$ over a fixed variety $X$. It may be identified as a psh metric $\Phi$ over the trivial model $(X\times\overline\bbd^*,L\times\overline\bbd^*)$, by setting $\Phi_z=\phi_{-\log|z|}$. In this case, the logarithmic growth condition is merely the usual linear growth condition on psh rays.
\end{example}

\noindent We then have as an immediate Corollary:
\begin{corollary}The space $\PSH(L)$ is stable under limits of decreasing nets, finite maxima, and addition of constants. It is furthermore the smallest such set containing all psh metrics on $L$ which admit a locally bounded extension to some model $(\cX,\cL)$ of $(X,L)$.
\end{corollary}
\begin{proof}
All of those properties are seen to preserve characterization (iv) above, having fixed some reference model. To show that it is the smallest set closed under those operations, only the statement about decreasing nets could \textit{a priori} be delicate. Given a metric $\phi\in\PSH(L)$, (i) shows that it extends as a genuine metric on some model $(\cX,\cL)$, and Demailly's regularization Theorem yields a decreasing sequence of smooth (in particular locally bounded) psh metrics decreasing to the extension of $\phi$, which shows in particular that $\phi$ belongs to the closure of the set of locally bounded psh metrics on $\cL$, proving our result.
\end{proof}

\subsection{The main setting, and some important examples.}

We begin with some notation. Let $\pi:X\to\pdisc$ be a degeneration together with a relatively ample line bundle $L$. We now, and for the remainder of this article, fix some reference boundary data $\phi_\partial$, which is the restriction to the boundary $\pi\mi(\bbs^1)$ of a smooth psh metric on $L$. This is a minor distinction which will allow us to later obtain a genuine metric structure on a particular subspace of $\PSH(L)$, rather than a pseudometric structure, and therefore we will assume that a metric in $\PSH(L)$ has boundary data equal to $\phi_\partial$. We will define $$\cE^1(L)=\cE^1_{X/\pdisc}(L)\cap\PSH(L)$$ to be the space of fibrewise finite-energy metrics in $\PSH(L)$. We also set
$$\hat\cE^1(L)=\{\phi\in\cE^1(L),\,\phi\text{ is relatively maximal}\}.$$

\begin{example}Although those are seemingly restrictive conditions, they are in fact general enough to encompass the study of maximal geodesic rays. Let $(X,L)$ be a product family $(M\times\pdisc,L_M\times\pdisc)$. For a $\bbs^1$-invariant metric $\phi$ on $L_M\times\pdisc$, seen as a ray $[0,\infty)\ni t\mapsto \phi_t,$
\begin{enumerate}
\item being in $\PSH(L)$ corresponds to the usual linear growth condition;
\item being relatively maximal corresponds to being a geodesic ray in the sense of \cite{bbj};
\item being in $\cE^1(L)$ corresponds to having fibrewise finite-energy and linear growth;
\item therefore, belonging to $\hat\cE^1(L)$ corresponds to being a fibrewise finite-energy geodesic rays with linear growth emanating from a given point - exactly the space of rays $\cR^1(L)$ considered in \cite{darlurays}.
\end{enumerate}
\end{example}

\begin{example}[Relative dimension zero, part 1]\label{exa_reldim}Consider the case of relative dimension zero with a trivial line bundle $L$ over $X\simeq \bbd^*$. Then,
\begin{enumerate}
\item $\PSH(L)$ corresponds to the set of subharmonic functions with logarithmic growth on $\bbd^*$;
\item the class of relatively maximal metrics in $\PSH(L)$ corresponds to the class of harmonic functions;
\item $\cE^1(L)$ corresponds to finite-valued subharmonic functions with logarithmic growth;
\item finally, $\hat\cE^1(L)$ corresponds to finite-valued harmonic functions with logarithmic growth.
\end{enumerate}
It well-known that any harmonic function on the punctured disc decomposes as a sum of a multiple of $\log|z|$ and the real part of an analytic function. This is where our general setting starts diverging from the better-behaved $\bbs^1$-invariant. Indeed, by \cite[Proposition 4.1]{bbj}, for rays of metrics of finite energy, maximality implies linear growth. However, in our case, maximality plus finite energy no longer implies logarithmic growth, since there exist harmonic functions on the punctured disc that do not have logarithmic growth at zero (e.g. the real part of $z\mapsto e^{1/z}$)!

\bsni Assuming logarithmic growth, we then have a full description of $\hat\cE^1(L)$ in relative dimension zero, since we then see that any (finite-valued) harmonic function with logarithmic growth has to be of the form $c\cdot \log|z|+H(z)$, where $H(z)$ is the solution of the generalized Dirichlet problem over the whole disc with the given boundary data. In particular, it is an affine space isomorphic to $\bbr$! This agrees with the radial case, where $\cE^1(L)$ is simply the set of affine functions on $[0,\infty)$ emanating from the same point, which is isomorphic to the set of possible slopes.
\end{example}

\begin{example}[Relative dimension zero, part 2]\label{exa_reldim2}
We now consider what will be a model case for many future considerations: we still work in relative dimension zero over $\pdisc$, but we choose a nontrivial line bundle on $\pdisc$. The existence of a model for $(\pdisc,L)$ means that there is a relatively ample extension $\cL\to\disc$. We can now pick a trivialization $\tau$ of $\cL$ over $\disc$, which allows us to identify a metric $\phi\in\PSH(L)$ (extended to $\cL$ via the logarithmic growth condition!) with the function
$$u=-\log|\tau|_\phi$$
on $\disc$. By the discussion above, if $dd^c\phi=0$, then $u$ decomposes as
$$u(z)=c\cdot \log|z|+H(z),$$
where $H$ is bounded on $\disc$. This decomposition (in particular, $c$ and $H$) depends on $\tau$; but the fact that $\phi$ can be decomposed in any trivialization in such a way does not!
\end{example}

\noindent This is a nice model case for us, because the Deligne pairing construction (in our setting of fibrations over $\pdisc$) naturally gives line bundles over $\pdisc$, as we see in action now.


\begin{corollary}\label{coro_trivialization}The relative maximality condition for metrics in $\cE^1(L)$ can be pushed forward to the base via the Deligne pairing, i.e. we have a well-defined map
$$\hat\cE^1(L)\to \hat\cE^1(\langle L^{d+1}\rangle_{X/\pdisc}).$$
Furthermore, a metric $\phi\in\cE^1(L)$ belongs to $\hat\cE^1(L)$ if and only if, for any model $(\cX,\cL)$ of $(X,L)$ and any trivialization of the Deligne pairing $\langle \cL^{d+1}\rangle_{\cX/\disc}$, denoting $u=-\log|\tau|_{\langle \phi^{d+1}\rangle_{\cX/\disc}}$, one has
$$u(z)= c\cdot \log|z|+H(z),$$
where $c$ is a real constant and $\mathrm{H}$ is a harmonic function on $\disc$ depending only on $\tau$ and the boundary data.
\end{corollary}
\begin{proof}
The map above is naturally given by
$$\phi\mapsto \langle \phi^{d+1}\rangle_{X/\pdisc},$$
in which case both statements are corollaries of Proposition \ref{prop_harmonicitymaximaldiscs} and the two examples above.
\end{proof}

\subsection{Metrization.}

As an important, and somewhat surprising consequence of our previous results, we may define a metric structure on the space $\hat\cE^1(L)$. This generalizes e.g. \cite{darlurays}, in which the authors endow the space of maximal psh rays with the distance
$$\hat d_1(\phi_0,\phi_1)=\lim_{t\to\infty}\frac{d_1(\phi_{0,t},\phi_{1,t})}{t}.$$
In the next Sections, we will show that this structure furthermore satisfies some good properties, namely completeness and geodesicity.
\begin{theorem}\label{thm_1mainthmsepdiscs}
The space $\hat\cE^1(L)$ can be endowed with a metric space structure, defined by the generalized slope $\hat d_1(\phi_0,\phi_1)$.
\end{theorem}

\noindent Naturally, this suggests that the $d_1$-distance is subharmonic with logarithmic growth along metrics in $\hat\cE^1(L)$, a fact that we prove now.
\begin{prop}\label{prop_distancesubharmonic}
Let $\phi_0$, $\phi_1\in\hat\cE^1(L)$. Then, the map
$$z\mapsto d_1(\phi_{0,z},\phi_{1,z})$$
is subharmonic with logarithmic growth on $\bbd^*$.
\end{prop}
\begin{proof}
By the formula for $d_1$,
$$d_1(\phi_{0,z},\phi_{1,z})=\langle\phi_{0,z}^{d+1}\rangle+\langle\phi_{1,z}^{d+1}\rangle-2\langle P(\phi_{0,z},\phi_{1,z})^{d+1}\rangle.$$
By Proposition \ref{prop_harmonicitymaximaldiscs}, the first two metrics on the right-hand side have zero curvature, therefore we are left to show that the metric $\langle P(\phi_{0},\phi_{1})^{d+1}\rangle_{X/\pdisc}$ is superharmonic. We pick any zero curvature metric $\refmetric$ on $\bbd^*$, and note that $\langle P(\phi_{0},\phi_{1})^{d+1}\rangle_{X/\pdisc}$ is superharmonic if and only if $\langle P(\phi_{0},\phi_{1})^{d+1}\rangle_{X/\pdisc}-\refmetric$ is a superharmonic function. Fix $a\in\bbd^*$ and let $r>0$ be such that $\bbd(a,r)=\{|z-a|\leq r\}\subset \bbd^*$. Let $\psi$ be the relatively maximal psh metric on $\bbd(a,r)$ and with boundary data 
\begin{equation}\label{eq_lem_dist_1}
\phi(z)=P(\phi_{0,z},\phi_{1,z}),\,\forall z\in S(a,r).
\end{equation}
Such a metric is given by Theorem \ref{thm_existencemaximal}. We now deduce the two following facts:
\begin{itemize}
\item[(i)] by maximality of $\psi$, if follows from Proposition \ref{prop_harmonicitymaximaldiscs} that $z\mapsto \langle \psi^{d+1}\rangle_{X/\pdisc}$ has zero curvature;
\item[(ii)] since on the boundary $S(a,r)$ we have $\psi(z)\leq \phi_{0,z},\phi_{1,z}$, and $\phi_{0}$, $\phi_{1}$ are relatively maximal, we have
$$\psi_z\leq \phi_{0,z},\phi_{1,z}$$
for all $z\in\bbd(a,r)$, thus $\psi_z\leq P(\phi_{0,z},\phi_{1,z})$ and finally
$$\langle\psi_z^{d+1}\rangle\leq \langle P(\phi_{0,z},\phi_{1,z})^{d+1}\rangle$$
by monotonicity of the Monge-Ampère energy.
\end{itemize}
Using (\ref{eq_lem_dist_1}), (i), and (ii) in order, we find: 
\begin{align*}
\fint_{S(a,r)}\langle P(\phi_{0,z},\phi_{1,z})^{d+1}\rangle - \phi_{\mathrm{ref},z}&=\fint_{S(a,r)} \langle \psi_z^{d+1}\rangle - \phi_{\mathrm{ref},z}\\
&=\langle \psi_a^{d+1}\rangle - \phi_{\mathrm{ref},a}\\
&\leq \langle P(\phi_{0,a},\phi_{1,a})^{d+1}\rangle - \phi_{\mathrm{ref},a}.
\end{align*}
As the inequality is true for all $a$, our metric $\langle P(\phi_{0},\phi_{1})^{d+1}\rangle_{X/\pdisc}$ is then superharmonic.

\bsni We now show that there exists a real number $a\in\bbr$ such that
$$z\mapsto d_1(\phi_{0,z},\phi_{1,z})+a\log|z|$$
is bounded above. By Lemma \ref{lem_extpshtame}(iv), for any model $(\cX,\cL)$ of $(X,L)$, fixing a reference metric $\refmetric\in\hat\cE^1(L)$ which is locally bounded on $\cL$, one has (up to adding large enough constants)
$$\phi_0\leq\refmetric + c\cdot\log|z|$$
for some real constant $c$. In this case,
$$d_1(\phi_{0,z}-c\cdot \log|z|,\phi_{\mathrm{ref},z})=\langle(\phi_{0,z}-c\cdot \log|z|)^{d+1}\rangle-\langle\phi_{\mathrm{ref},z}^{d+1}\rangle,$$
and the term on the right-hand side is a harmonic function with logarithmic singularities at the origin, so that substracting constants the result also holds for $z\mapsto d_1(\phi_{0,z},\phi_{\mathrm{ref},z})$. Proceeding similarly for $\phi_1$, our result then follows from the triangle inequality.
\end{proof}

\noindent Finally, we note an immediate consequence of Lemma \ref{lem_lelongestimate} together with the previous Proposition \ref{prop_distancesubharmonic}.
\begin{lemma}\label{lem_unifcontrol}
Let $\phi_0$, $\phi_1\in\hat\cE^1(L)$. Then, for all $z$ on the base, we have
$$d_1(\phi_{0,z},\phi_{1,z})\leq \hat d_1(\phi_0,\phi_1)\log(1/|z|).$$
\end{lemma}
\begin{remark}Had we not fixed boundary data, we would have an additional error term in the above expression, corresponding exactly to the supremum of $z\mapsto d_1(\phi_{0,z},\phi_{1,z})$ for $z\in\bbs^1$.
\end{remark}

\noindent We are now equipped to endow the space $\hat\cE^1(L)$ with a metric structure. 
\begin{proof}[Proof of Theorem \ref{thm_1mainthmsepdiscs}.]
That $\hat d_1(\phi,\phi)=0$ and $\hat d_1(\phi_0,\phi_1)=\hat d_1(\phi_1,\phi_0)$ are immediate statements, and nonnegativity will follow from the triangle inequality and the former statement. Therefore, we must show that for any other $\phi_2\in\hat\cE^1(L)$, we have
$$\hat d_1(\phi_0,\phi_1)\leq \hat d_1(\phi_0,\phi_2)+\hat d_1(\phi_2,\phi_1).$$
Let $a_{01}$ be such that $d_1(\phi_{0,z},\phi_{1,z}) + a_{01}\log|z|$ is bounded above on the punctured disc, and define similarly $a_{02}$, $a_{21}$. We have by the triangle inequality of the fibrewise metric $d_1$
$$d_1(\phi_{0,z},\phi_{1,z})\leq d_1(\phi_{0,z},\phi_{2,z})+d_1(\phi_{2,z},\phi_{1,z})$$
for all $z$ in $\bbd^*$, and in particular
$$d_1(\phi_{0,z},\phi_{1,z})+(a_{02}+a_{21})\log|z|\leq d_1(\phi_{0,z},\phi_{2,z})+d_1(\phi_{2,z},\phi_{1,z})+(a_{02}+a_{21})\log|z|.$$
Upon taking (negative) Lelong numbers and adding constants, we find
\begin{align*}&a_{02}+a_{21}-\nu_0(d_1(\phi_{0,z},\phi_{1,z})+(a_{02}+a_{21})\log|z|)\\
&\leq a_{02} -\nu_0(d_1(\phi_{0,z},\phi_{2,z})+a_{02}\log|z|)+a_{21} -\nu_0(d_1(\phi_{2,z},\phi_{1,z})+a_{21}\log|z|).
\end{align*}
Since
$$z\mapsto d_1(\phi_{0,z},\phi_{1,z})+(a_{02}+a_{21})\log|z|$$
is bounded above, the previous equation is by the very definition of $d_1$ equivalent to
$$\hat d_1(\phi_0,\phi_1)\leq \hat d_1(\phi_0,\phi_2)+\hat d_1(\phi_2,\phi_1),$$
as desired. Finally, assuming $\hat d_1(\phi_0,\phi_1)=0$, Lemma \ref{lem_unifcontrol} shows that we must have $\phi_0=\phi_1$.
\end{proof}

\subsection{Completeness.}

We now prove completeness of our space.
\begin{theorem}\label{thm_1completeness}
The metric space $(\hat\cE^1(L),\hat d_1)$ is complete.
\end{theorem}
\noindent In order to prove this, discuss possible topologies for $\cE^1(L)$.
\begin{remark}[Topologies on $\cE^1(L)$]We have already considered the topology of fibrewise $d_1$-convergence on $\cE^1(L)$. There is a yet finer topology, that of locally uniform fibrewise $d_1$-convergence, by which $\phi_k$ converges to $\phi$ if, for all relatively compact open sets $U$ in $X$, $d_1(\phi_{k,z},\phi_z)\to 0$ uniformly in $z$ on $U$. In between the two, there is the topology of "base-locally" uniform fibrewise $d_1$-convergence, which is the same but over the $\pi\mi(U)$ with $U$ relatively compact open in $\bbd^*$. By the previous Lemma, the latter is equivalent to the topology induced by $\hat d_1$ on $\hat\cE^1(L)$!
\end{remark}

\begin{prop}\label{prop_closed}Let $\phi_k$ be a sequence of metrics in $\hat\cE^1(L)$ converging to some metric $\phi$ on $L$ for the topology of base-locally uniform fibrewise $d_1$ convergence. Then, $\phi$ belongs to $\hat\cE^1(L)$.
\end{prop}
\begin{proof}
Pick a sequence $k\mapsto \phi_k\in\hat\cE^1(L)$ and a fixed metric $\phi$ in $\cE^1(L)$. Assume that, for a relatively compact open $U\subset\pdisc$ we have
$$d_1(\phi_{k,z},\phi_z)\to 0$$
uniformly in $z\in \pi\mi(U)$. Since convergence in Monge-Ampère energy is subordinate to $d_1$-convergence we have that
$$\langle\phi_{k,z}^{d+1}\rangle\to \langle\phi_z^{d+1}\rangle$$
again uniformly in $z$; by maximality, the metrics $\langle\phi_{k}^{d+1}\rangle_{X/\pdisc}$ are zero curvature, and an uniform limit of such is again zero curvature. As having zero curvature is a local property and the $\pi\mi(U)$ cover $X$, we then have that $\langle\phi^{d+1}\rangle_{X/\pdisc}$ has zero curvature on all of $X$. By virtue of being in $\cE^1(L)$, this implies $\phi$ to be relatively maximal by \ref{prop_harmonicitymaximaldiscs}, as long as we can show that the limit is psh. On $\pi\mi(U)$, there exists $c>0$ independent of $z\in U$ such that
$$\int (\phi_{k,z}-\phi_z)\,d\mu_z \leq c\cdot d_1(\phi_{k,z},\phi_z)\leq c\cdot c'$$
against a fixed smooth family of volume forms $z\mapsto \mu_z$, so that uniform fibrewise $d_1$-convergence implies $L^1$ convergence of $\phi_{k}$ to $\phi$ on $\pi\mi(U)$, which establishes plurisubharmonicity of the limit there, hence on $X$.
\end{proof}

\bsni We may now prove completeness.
\begin{proof}[Proof of Theorem \ref{thm_1completeness}.]
Consider a Cauchy sequence $m\mapsto \phi_m\in\hat\cE^1(L)$. For all $\varepsilon$ and all large enough $m$, $n$, 
$$\hat d_1(\phi_m,\phi_n)\leq \varepsilon,$$
which by Lemma \ref{lem_unifcontrol} implies the individual sequences $m\mapsto \phi_{m,z}$ to be $d_1$-Cauchy. By completeness of the fibrewise $\cE^1$ spaces (\cite[Theorem 3.36]{darbook}), those sequences $d_1$-converge to a unique finite-energy metric $\phi(z)$, and in fact this convergence is seen to hold base-locally uniformly fibrewise. The mapping
$$z\mapsto \phi(z)$$
is therefore a metric in $\hat\cE^1(L)$ by Proposition \ref{prop_closed}.
\end{proof}

\subsection{Geodesics.}

We now show that, much as in the absolute $\cE^1$ setting, one can find geodesics in $\hat\cE^1(L)$. 

\begin{theorem}\label{thm_1geodesics}
Given any $\phi_0$, $\phi_1\in\hat\cE^1(L)$, the psh geodesic segment $t\mapsto \phi_t$ joining them, given by Theorem \ref{thm_relsegments}, is $\hat d_1$-geodesic in the metric sense, i.e.
$$\hat d_1(\phi_t,\phi_s)=|t-s|\hat d_1(\phi_0,\phi_1).$$
Furthermore, given any model $(\cX,\cL)$ of $(X,L)$ and a trivialization $\tau$ of $\langle \cL^{d+1}\rangle$ over $\disc$, setting
$$u_t:=-\log|\tau|_{\phi_t},$$
the segment of generalized slopes
$$t\mapsto \hat u_t$$
is affine on $[0,1]$; and $t\mapsto \phi_t$ is uniquely characterized by this property among psh segments.
\end{theorem}

\begin{proof}[Proof of Theorem \ref{thm_1geodesics}.]
Let $\phi_0$, $\phi_1\in\hat\cE^1(L)$. We consider as in Theorem \ref{thm_relsegments} the family of fibrewise maximal geodesics
$$t\mapsto \phi_{t,z}.$$
To show that it belongs to $\hat\cE^1(L)$, we must make sure that it has logarithmic growth and is relatively maximal. The former is due to Lemma \ref{lem_extpshtame}(ii), since for fixed $x\in X$, $\phi_t(x)\leq (1-t)\phi_0(x)+t\phi_1(x)$ by convexity of maximal segments, so that if there exist $a_i$, $i=0,1$ such that $\phi_i + a_i\log|z|$ are bounded above near the central fibre of some model, then so is $\phi_t + (1-t)a_0+a_1$. Regarding maximality, $\langle \phi_t^{d+1}\rangle_{X/\pdisc}$ is a convex combination of zero curvature metrics with logarithmic growth, hence $\phi_t$ is also relatively maximal by Proposition \ref{prop_harmonicitymaximaldiscs}.

\bsni That $t\mapsto \phi_t$ is $\hat d_1$-geodesic is a consequence of the fact that, for all $z$, $t\mapsto \phi_{t,z}$ is $d_{1,z}$-geodesic. Finally, the statement regarding the Monge-Ampère energy follows upon taking generalized slopes in the statement of Theorem \ref{thm_relsegments}.
\end{proof}

\subsection{Extension of the distance to $\cE^1(L)$.}

In this Section, we construct a "maximal envelope" map, which will allow us to extend the $d_1$-distance as a pseudodistance to all of $\cE^1(L)$.

\begin{prop}\label{prop_envelope}For all $\phi\in\cE^1(L)$, there exists a unique smallest relatively maximal metric $\hat P(\phi)\in\hat\cE^1(L)$ with $\phi\leq\hat P(\phi)$ and
$$d_1(\phi_z,\hat P(\phi)_z)=o\,(\log|z|)$$
as $z\to 0$. This defines a natural projection
$$\hat P:\cE^1(L)\to \hat\cE^1(L).$$
\end{prop}

\noindent Before proving this result, we note this immediate Corollary:
\begin{corollary}The mapping
$$\hat d_1(\phi_0,\phi_1)=\hat d_1(\hat P(\phi_0),\hat P(\phi_1))$$
defines a pseudodistance on $\cE^1(L)$.
\end{corollary}

\begin{proof}[Proof of Proposition \ref{prop_envelope}.]
Let $\phi\in\cE^1(L)\cap C^0(L)$, and, for all $r\in (0,1)$, let $U_r$ denote the annulus $\{r < |z| < 1\}\subset \pdisc$, and $V_r=\pi\mi(U_r)\subset X$. Let $\phi_r$ be the relatively maximal metric on $V_r$, coinciding with $\phi$ on $\partial V_r$, given by Theorem \ref{thm_existencemaximal}. Fixing $z$ on the base, the sequence $r\mapsto \phi_{r,z}$ is an increasing sequence of psh metrics in $\cE^1(L_z)$. We claim that the limit family
$$z\mapsto \left(\lim_{r\to 0}{}^* \phi_{r,z}\right)$$
is the desired envelope $\hat P(\phi)$. Denote this limit $\hat \phi$ for the moment. Fix some $r$. By construction, $\hat \phi$ restricted to $V_r$ coincides everywhere with its Perron-Bremmermann envelope; furthermore, it is locally bounded (since it is approximable from below). By the discussion in Section \ref{sect_relmax}, since this holds for all $r$, $\hat \phi$ is relatively maximal. Furthermore, by construction again, it satisfies $\phi\leq\hat\phi$ and is the smallest such relatively maximal metric. We are therefore only left to prove that $d_1(\hat\phi_z,\phi_z)=O(\log|z|)$ as $z\to 0$. As in Corollary \ref{coro_trivialization}, we pick a model $(\cX,\cL)$ of $(X,L)$, and we extend $\langle \phi^{d+1}\rangle$ to the trivializable line bundle $\langle \cL^{d+1} \rangle$. Picking a trivialization $\tau$ allows us to identify the energies $\langle\hat \phi^{d+1}\rangle$ and the $\langle\phi_r^{d+1}\rangle$ with functions $u$ and $u_r$ on $\bbd$ and $U_r$ respectively. By Proposition \ref{prop_harmonicitymaximaldiscs}, those functions are harmonic, and for all $s\in(0,1)$, the functions $u_r$, $r>s$ increase over $\overline{U_s}$ to $u$, which implies the convergence to be uniform (as an increasing sequence of harmonic functions over a compact set). Now, by harmonicity, for $r>s$, the integrals
$$\fint_{|z|=r}u_s(z)\,dz$$
are affine functions of $\log r$. Writing
$$v=-\log|\tau|_{\langle \phi^{d+1}\rangle},$$
we then have
$$\fint_{|z|=r}u_s(z)\,dz=\frac{\log r}{\log s}\fint_{|z|=s}v(z)\,dz+\left(1-\frac{\log r}{\log s}\right)\cdot \fint_{|z|=1}v(z)\,dz,$$
(recall how we have defined $\phi_s$ and $u_s$!). Taking the limit $s\to 0$ using the uniform convergence discussed above yields
$$\fint_{|z|=r}u_s(z)\,dz=-(\log r)\, \hat v + \fint_{|z|=1}v(z)\,dz,$$
where $\hat v$ denotes the generalized slope of the subharmonic function $v$. Taking slopes in this equality, one then finds
$$\hat v = \hat u.$$
Now, since $\phi\leq\hat\phi$, we have
$$d_1(\phi_{z},\hat \phi_z)=u(z)-v(z),$$
whose slopes we have seen to coincide, proving our statement that $d_1(\hat\phi_z,\phi_z)=O(\log|z|)$. Therefore, $\hat \phi$ is our desired envelope $\hat P(\phi)$. Finally, if $\phi$ is not continuous, we extend it to some model $(\cX,\cL)$, and a decreasing approximation by continuous metrics $\phi_i$ on $\cL$ gives a sequence of relatively maximal metrics $\hat \phi_i$ decreasing to some relatively maximal metric $\hat \phi$ which has the desired properties, as we show now: define $u^i$, $u^i_r$ and $v^i$ as above for $\phi_i$, and $u$, $u_r$ and $v$ for $\phi$. By monotonicity of Deligne pairings along decreasing nets, we have that $u^i\to u$, $u^i_r\to u_r$ and $v^i\to v$ decreasingly; we then have for all $r>s\in(0,1)$ and all positive integers $i$ that
$$\fint_{|z|=r}u^i_s(z)\,dz=\frac{\log r}{\log s}\fint_{|z|=s}v^i(z)\,dz+\left(1-\frac{\log r}{\log s}\right)\cdot \fint_{|z|=1}v^i(z)\,dz;$$
furthermore, we may normalize all our sequences so that all the functions involved are nonpositive, thereby allowing us to use monoton convergence and find
$$\fint_{|z|=r}u_s(z)\,dz=\frac{\log r}{\log s}\fint_{|z|=s}v(z)\,dz+\left(1-\frac{\log r}{\log s}\right)\cdot \fint_{|z|=1}v(z)\,dz,$$
so that we may proceed using the same argument as before to show that $d_1(\hat\phi_z,\phi_z)=O(\log|z|)$; that $\hat \phi$ is the smallest relatively maximal metric bounded below by $\phi$ and satisfying this equality follows again by construction, since decreasing limits of relatively maximal metrics over annuli remain relatively maximal.
\end{proof}

\section{The non-Archimedean limit.}\label{sect_sect2}

We move away from relatively maximal and finite-energy metrics for the moment, and focus on the space $\PSH(L)$. The purpose of this Section is to show that there is a natural map from this space to a certain space of non-Archimedean metrics. We describe the non-Archimedean setting in Sections \ref{subsect_21} and \ref{sect_nappt}. We prove some complex preparations in Section \ref{subsect_23}, then describe the construction in Section \ref{subsect_24}. Finally, in Section \ref{subsect_25}, we show how a certain subclass of metrics behaves under this map.

\subsection{Degenerations as varieties over discretely valued fields.}\label{subsect_21}

Dating back to ideas of Berkovich (\cite{berkvanishing}, \cite{berkhodge}), objects such as degenerations and analytic models thereof can be interpreted as varieties over the field $\bbc((t))$ (see also \cite{favre}, \cite{bjtrop}). For clarity, we will from now on write $\K=\bbc((t))$ and $\R=\bbc[[t]]$.

\bsni Pick a degeneration $\pi:X\to\pdisc$ and an analytic model $\pi:\cX\to\disc$ of $X$. As $X$ is projective, it can be embedded in some $\bbp^n\times\bbd$, where it is presented by a finite number of homogeneous polynomials with coefficients in the set of holomorphic functions on $\pdisc$ that are meromorphic at zero. Since this set of functions can be identified with the field $\K$ of complex Laurent series, one can then view $X$ as a variety $X_\K$ over the field $\K$. Similarly, $\cX$ can be presented by finitely many homogeneous polynomials with coefficients in $\cO(\bbd)$, i.e. holomorphic functions over the disc, so that it can be identified with a variety $\cX_\R$ over $\R$.

\begin{example}In the case of an isotrivial degeneration $X\simeq M\times \bbd^*$ for some complex projective manifold $M$, $X$ can be identified with the base change of $M$ to the field $\K$. In particular, there exists a "trivial" algebraic model, defined by taking the base change of $M$ to $\R$, which corresponds to the product analytic family over $\bbd$.
\end{example}

\noindent $\K$ is a (non-Archimedean) valued field, with valuation
$$\nu_0(\Sigma a_i t^i)=\min\{i,\,a_i\neq 0\}.$$
This also defines a valuation on the Noetherian ring $\R$. From the general work of Berkovich (\cite{berko}), one can associate to a scheme $X$ over a valued ring $\R$, in a functorial way, its analytification $X\an$ with respect to the given valuation on the base. The underlying points of this analytification roughly correspond to valuations on the function field $\mathrm{K}(X)$ extending the base valuation on $\mathrm{K}$, and the topology is that of pointwise convergence. 

\bsni In our setting, the Berkovich analytification $X_\K\an$ of $X_\K$ contains an important dense subset: the set of divisorial points $X^{\mathrm{div}}$. It is described as follows. Let $\cX$ be an analytic model of $X$. By Noetherianity and normality, the fibre of $\cX$ over $0$ is then a Cartier divisor which decomposes as the Weil divisor
$$\cX_0=\sum_i a_i E_i,$$
with each $E_i$ irreducible. Each component of such a decomposition defines a valuation $\nu_{E_i}$ on $\K(X)$ as follows: for all $f\in \K(X)$,
$$\nu_{E_i}(f)=\ord_{E_i}(f)/a_i.$$
All divisorial points of $X_\K\an$ are then obtained in this manner.

\subsection{Non-Archimedean plurisubharmonic functions.}\label{sect_nappt}

Let $X$ be a degeneration with a line bundle $L$ on $X$. Let $(\cX,\cL)$ be a model of $(X,L)$. To $\cL$ one can associate a model metric $\phi_\cL$ on $L_\K\an$, as explained in detail in \cite{bfjsemi}. Such a metric is uniquely characterized as follows: given an open set $\mathcal{U}\subset\cX$ and a nonvanishing section of the restriction of $\cL$ to $\mathcal{U}$, then we require that $|s|_{\phi_\cL}=1$ on $(\mathcal{U}_\K\cap X_\K)\an$.

\begin{defi}We say that a model metric $\phi_\cL$ is plurisubharmonic if $\cL$ is nef. Given a metric $\phi$ on $L_\K\an$, we say that it is plurisubharmonic, and we write $\phi\in\PSH(L_\K\an)$ if there exists a sequence of plurisubharmonic model metrics on $L_\K\an$ decreasing to $\phi$.
\end{defi}

\noindent Fixing a psh model metric $\phi_\cL$ on $L_\K\an$, one can identify psh metrics on $L_\K\an$ with "$\cL$-psh" functions on $X_\K\an$, via $\phi \leftrightarrow \phi-\phi_\cL$. We define more generally the set of $L$-psh functions to be the reunions of all $\cL$-psh functions for all nef models $\cL$ of $L$.

\bsni We usually endow the space of $\cL$-psh functions with the topology of pointwise convergence on divisorial points, i.e. $\phi_k\to\phi$ in $\PSH(\cL_\R\an)$ if and only if, for all $\nu\in X^{\mathrm{div}}$, $\phi_k(\nu)\to\phi(\nu)$. We note (\cite{bfjsemi}) that a non-Archimedean psh function is uniquely determined by its values on $X^{\mathrm{div}}$!

\bsni Any vertical ideal sheaf $\mathfrak{a}$ on a model $\cX$ of $X$ defines a function $\log|\mathfrak{a}|$ on $X$, via
$$\log|\mathfrak{a}|(x)=\max \{\log|f(x)|\},$$
where the $f$ run over a set of local generators for $\mathfrak{a}$. (In particular, any vertical Cartier divisor $D$ on a model defines such a function.) We then have the following crucial result:
\begin{lemma}[{\cite{bfjsemi}}]
Let $(\cX,\cL)$ be a model of $(X,L)$. Let $\mathfrak{a}$ be a vertical ideal sheaf on $\cX$, such that $\cL\otimes\mathfrak{a}$ is globally generated. Then, $\phi_\cL + \log|a|$ is a psh metric on $L_\K\an$.
\end{lemma}

\subsection{The main result.}

We are now equipped to describe the main construction of this Section. We fix a metric $\phi\in\PSH(L)$. Given any divisorial point $\nu_E$ associated to the component $E$ of a model $\cX$ of $X$, we know that $\phi+a\log|z|$ extends to a metric over $E$ for some $a\in\bbr$. Pick a psh metric $\phi_{E}$ with divisorial singularities of type $E_i$ on $\cX$, i.e. locally of the form 
$$\phi_{E}=\log |f_{E}| + O(1),$$ 
where $f_{E}$ is a local equation for $E$. We can then define a generic (signed) Lelong number
\begin{equation}\label{eq_genericlelong}
\varphi\na(\nu_E)=\ord_E(\phi):=-\sup\{c\geq 0,\,\phi+a\log|z|\leq c\cdot \phi_{E}+O(1)\text{ near E}\}+a.
\end{equation}
By linearity, this is independent of the choice of such an $a$. Performing this construction over all possible $E$ captures the singularities of $\phi$ along all possible models of $\cX$. Our main result for this Section is then the following:
\begin{theorem}\label{thm_namap}
Let $X$ be a degeneration together with a relatively ample line bundle $L$. The Lelong numbers of a metric $\phi\in\PSH(L)$ define a function on $X^{\mathrm{div}}$, which admits a unique $L$-psh extension, giving a map 
$$(\cdot)\na:\PSH(L)\to \PSH(L_\K\an),$$
which is furthermore lower semicontinuous and order-preserving.
\end{theorem}

\subsection{Some preliminaries.}\label{subsect_23}

We now prove some auxiliary results that will be useful in the proof of Theorem \ref{thm_namap}. We first show that multiplier ideals of psh metrics on $L$ give $L_\K\an$-psh functions.

\begin{lemma}\label{lem_multidealisvertical}
Let $\phi$ be a metric in $\cE^1(L)$. Let $(\cX,\cL)$ be a model of $(X,L)$ such that $\phi$ extends as a psh metric on $\cL$. Then, up to restricting to a slightly smaller disc, for all $m$, the multiplier ideal $$\ideal_m=\cJ(m\phi)$$ is vertical, and there exists an integer $m_0$ (depending only on $\cL$ and not on $m$ or $\phi$) such that $(m+m_0)\cL\otimes \cJ(\phi)$ is globally generated on $\cX$.
\end{lemma}
\begin{proof}
Since $\phi$ in particular has fibrewise finite energy, it has zero Lelong numbers on all fibres. As a consequence, $\phi$ has zero Lelong numbers on all of $\cX-\cX_0$, as Lelong numbers cannot increase upon evaluating them on a larger space. Skoda's integrability theorem (\cite[Theorem 1]{sko72}, see also \cite[Lemma 5.6(a)]{demanalytic}) then yields local $L^1$-integrability of $e^{-\phi}$, which in turn implies local $L^p$-integrability of $e^{-\phi}$ for all $\infty>p\geq 1$, and in particular, for all positive integers $m$, $L^1$-integrability of $e^{-m\phi}$. By \cite[Lemma 5.6(a)]{demanalytic} again, the multiplier ideals satisfy $$\ideal_m{}_{,x}=\cO_\cX{}_{,x}$$ for all $m$ and for all $x$ outside of the central fibre, i.e. $\ideal_m$ is cosupported on the central fibre.

\bsni Now, the global generation statement, follows from a relative equivalent of \cite[Proposition 6.27]{demanalytic}. We can in fact argue just as in \cite[Lemma 5.6]{bbj}: we must prove that there exists $m_0$ such that the sheaf $(m+m_0)\cL\otimes \cJ(\phi)$ is $\pi$-globally generated. By the relative Castelnuovo-Mumford criterion, having picked a relatively very ample line bundle $V$ on $X$ and an $m_0$ such that $m_0\cdot \cL - K_\cX - (d+1)V$ is relatively ample (after possibly restricting to a smaller disc), it is enough to show that for all $j=1,\dots,d$,
$$R^j\pi_*(((m+m_0)\cL-jV)\otimes \cJ(\phi))=0$$
on the disc, which follows from Kodaira and Nadel vanishing.
\end{proof}

\noindent We thus obtain the following:
\begin{corollary}\label{cor_multidealvertical}For any metric $\phi\in\PSH(L)$, and any model $(\cX,\cL)$ of $(X,L)$ such that $\phi$ extends as a psh metric on $\cL$, there exists an integer $m_0$ such that the function
$$(m+m_0)\mi\log |m\cJ(\phi)|$$
is $\cL$-psh for all positive integers $m$.
\end{corollary}
\begin{proof}In the case where $\phi$ also has fibrewise finite energy, this follows from the previous Lemma. In the general case, one can approximate $\phi$ on $\cL$ by a decreasing sequence of (e.g.) locally bounded metrics $\phi_k$. Since the integer $m_0$ depends only on $\cL$, the sequence
$$k\mapsto \tilde\phi_k:=(m+m_0)\mi\log |m\cJ(\phi_k)|$$
is then a sequence of $\cL$-psh functions. Since the sequence $\phi_k$ is decreasing, we have for all $k$ that $\cJ(\phi_{k+1})\subseteq \cJ(\phi_k)$, i.e. the sequence $\tilde\phi_k$ is also decreasing, which implies its limit $(m+m_0)\mi\log |m\cJ(\phi)|$ to be $\cL$-psh, as desired.
\end{proof}

\bsni We conclude our preliminaries by introducing the log discrepancy function
$$A_\cX:\cX^{\mathrm{div}}\to \bbr$$ of a model, which will be used in the proof of Theorem \ref{thm_namap}. Pick a model $\cX$ of $X$, and let $\rho:\cY\to\cX$ be some model dominating $\cX$. Any point in $cX^{\mathrm{div}}$ is a valuation $\nu_E$ associated to a divisor in the central fibre $\cY_0=\sum_i a_i E_i$ of such models, and the log discrepancy function is thus fully characterized via the formula
$$K_{\cY}+\cY_0 = \rho^*(K_\cX + \cX_0) + \sum_i A_\cX(\nu_{E_i})a_i E_i.$$

\subsection{Proof of Theorem \ref{thm_namap}.}\label{subsect_24}

We may now prove Theorem \ref{thm_namap}. 

\begin{proof}We fix a metric $\phi\in\PSH(L)$. We need to show that the function defined on $X^{\mathrm{div}}$ by
$$\phi\na:\nu_E\mapsto \ord_E(\phi),$$
where $\nu_E$ corresponds to a divisorial valuation and $\ord_E$ is defined as a generic Lelong number as in \eqref{eq_genericlelong}, admits a psh extension on $X_\K\an$. Since a non-Archimedean psh function is uniquely defined on the set of divisorial points, it is then though to show that $\phi\na$ can be approximated by a decreasing sequence of psh model functions on $X_\K\an$. Note that, by construction, the map $\phi\mapsto \phi\na$ is lsc and order preserving.

\bsni By Corollary \ref{cor_multidealvertical}, the metric
$$\psi_m=(m+m_0)\mi u_m,$$
where $u_m$ is the model function $\log|\cJ(m\phi)|$, is $\cL_\R\an$-psh. Pick a divisorial point $\nu_E\in X^{\mathrm{div}}$ associated to a component in the central fibre of an analytic model $(\cX,\cL)$ of $(X,L)$.  Using a version of the estimate \cite[Lemma B.4]{bbj} (which is proven exactly as in the trivially valued case), one has
$$m\cdot \varphi\na(\nu_E) \leq u_m(\nu_E) \leq m\cdot \varphi\na(\nu_E) + A_\cX(\nu_E),$$
where $A_\cX$ is the log discrepancy function as before. The sequence $\psi_m$ is therefore a sequence of $\cL_\R\an$-psh functions converging pointwise on $X^{\mathrm{div}}$ to $\varphi\na$. To show that $\varphi\na$ is $\cL_\R\an$-psh, it is then enough to prove that we can have this sequence be decreasing. By subadditivity of multiplier ideals we have $$\cJ(2m\phi)\subseteq \cJ(m\phi)^2,$$
thus
$$\psi_{2m}\leq 2\psi_{m},$$
and as $\phi_{m}\leq 0$,
$$\psi_{2m}\leq \frac{2(m+m_0)}{2m+m_0}\psi_m\leq \psi_m.$$
Picking the subsequence $i\mapsto \psi_{2^i}$ therefore yields a decreasing subsequence converging to $\varphi\na$, as desired. We then set
$$\phi\na:=\varphi\na + \phi_\cL,$$
which concludes our proof.
\end{proof}

\subsection{Locally bounded metrics in the non-Archimedean limit.}\label{subsect_25}

We now begin studying the behaviour under the map $(\cdot)\na$ of the class of metrics $\phi$, such that there exists a model $(\cX,\cL)$ of $(X,L)$ on which $\phi$ admits a locally bounded extension. 

\begin{prop}\label{prop_extension_on_model}
Let $\phi\in\PSH(L)$. Then, 
\begin{enumerate}
\item $\phi$ extends to a psh metric on a model $(\cY,\cM)$ of $(X,L)$ if and only if $\phi\na\leq \phi_\cM$;
\item $\phi$ extends to a locally bounded psh metric on $(\cY,\cM)$ if and only if $\phi\na=\phi_\cM$.
\end{enumerate}
\end{prop}
\begin{proof}
Note that it is equivalent to show the following: given $(\cX,\cL)$ an analytic model of $(X,L)$ and $\psi$ be a reference metric admitting a locally bounded extension to $\cL$,
(1) holds if and only if $\phi\na-\psi\na\leq \phi_\cM - \phi_\cL$, and (2) if and only if we have equality. This will allow us to work at the level of functions and relatively to another model, which is easier.
 
\bsni Assume first $\phi$ to extend to a psh metric on $\cM$. Let $\cZ$ dominate both models via $\pi_\cX:\cZ\to\cX$ and $\pi_\cY:\cZ\to \cY$, and we have $$\pi_\cY^*\cM=\pi_\cX^*\cL+D$$
for a unique Cartier divisor $D$ supported in the special fibre $\cZ_0$. Since $\phi$ extends to a psh metric on $(\cY,\cM)$ if and only if it extends to a psh metric on any model dominating $(\cY,\cM)$, we may without loss of generality focus on $\cZ$. Picking a local equation $f_D$ for the divisor $D$ obtained as above, $\phi$ extends to $\pi_\cY^*\cM$ if $$\phi-\psi\leq -\log|f_D| + C$$ 
near $\cZ_0$. Taking generic Lelong numbers with respect to the underlying divisor of a divisorial point $\nu$ gives
$$\nu(\phi)-\nu(\psi) \geq -\nu(D),$$
i.e. 
$$\phi\na(\nu)-\psi\na(\nu)\leq \phi_\cM(x)-\phi_\cL(x).$$
In the case where $\phi$ admits a locally bounded extension, then there is also a lower bound, which shows by the same argument that $\phi\na=\phi_\cM-\phi_\cL$. The converse is obtained by uniqueness of the Siu decomposition of $\phi$ on $\cX$.
\end{proof}


\section{Finite-energy spaces and the Monge-Ampère extension property.}\label{sect_sect3}

\subsection{Non-Archimedean finite-energy metrics.}\label{sect_naenergy}

We begin this Section with some reminders from non-Archimedean pluripotential theory. Let $X$ be a general variety over $\K$ endowed with an ample line bundle $L$. As in complex geometry, one can define Monge-Ampère measures associated to a tuple of $d=\dim X$ continuous non-Archimedean $L$-plurisubharmonic metrics. The general construction relies on intersection pairings (see e.g. \cite{gublerinventionestropical}, \cite{be20}), or Chambert-Loir and Ducros' theory of differential forms on Berkovich spaces (\cite[5, 6]{cld}), building on Lagerberg's theory of differential superforms. We only describe the main results below.

\bsni Given $d$ continuous psh metrics $\phi_1,\dots,\phi_d$ on $L\an$, we have a Radon probability measure
$$\MA(\phi_1,\dots,\phi_d)=V\mi\cdot dd^c\phi_1\wedge\dots\wedge dd^c\phi_d\wedge\delta_X,$$
where $V=(L^d)$. For ease of notation, we will also write
$$\MA(\phi)=V\mi\cdot dd^c\phi\wedge\dots\wedge dd^c\phi\wedge\delta_X=V\mi\cdot (dd^c\phi)^d\wedge\delta_X.$$
Just as in the complex case, one can define the space of finite-energy metrics $\cE^1(L\an)$, having extended the Monge-Ampère energy via decreasing limits again.

\bsni Using the results of \cite{reb20b}, one can also metrize $\cE^1(L\an)$  via setting
$$d_1(\phi_0,\phi_1)=E(\phi_0)+E(\phi_1)-2E(P(\phi_0,\phi_1)),$$
where $P(\phi_0,\phi_1)$ is the envelope
$$P(\phi_0,\phi_1)=\sup\,\{\phi\in \PSH(L),\,\phi\leq\min(\phi_0,\phi_1)\}.$$
This gives $\cE^1(L\an)$ a metric space structure which is furthermore geodesic, and which admits distinguished maximal geodesics characterized by the fact that the energy is affine along them. Finally, much as in the complex setting, one can use non-Archimedean Deligne pairings over a point (\cite{be20}) to realize the relative Monge-Ampère energy between two metrics in $\cE^1(L\an)$:
$$E(\phi_0)-E(\phi_1)=\langle \phi_0^{d+1}\rangle - \langle \phi_1^{d+1}\rangle.$$
Finally, much as in Section \ref{sect_deligne}, we note that we can extend the Deligne pairing construction over a point in the non-Archimedean case, to line bundles metrized by non-Archimedean finite-energy metrics.

\subsection{The Monge-Ampère energy in the non-Archimedean limit.}\label{subsect_32}

In the trivially-valued setting, we have already seen that a metric in $\cE^1(L)$ coincides with a finite-energy psh geodesic ray $t\mapsto \phi_t$. Two natural "asymptotic" energies arise:
\begin{enumerate}
\item the radial limit $\lim_t \frac{E(\phi_t)}{t}$;
\item the non-Archimedean energy of the non-Archimedean metric $\phi\na$ associated to $\phi$.
\end{enumerate}
In \cite{bbj}, it is established that if $\phi$ extends to a locally bounded metric on a test configuration, then those two quantities coincide. This is not the case in general, however. In this Section, we generalize those results to our relatively maximal psh metrics on degenerations. It will be clearer to express this using the relative dimension zero case of the construction from the previous Section.

\begin{remark}[Relative dimension zero and the non-Archimedean limit]\label{rem_namap}As mentioned in Example \ref{exa_reldim2} and Corollary \ref{coro_trivialization}, given a model $(\cX,\cL)$ of $(X,L)$ and a metric $\phi\in\hat\cE^1(L)$, one can identify the Monge-Ampère energy $\langle \phi^{d+1}\rangle_{X/\pdisc}$ of $\phi$ with a function on the punctured disc, by picking a trivialization $\tau$ of $\langle \cL^{d+1}\rangle$ and setting $u=-\log|\tau|_{\phi}$. The function $u$ then has a finite generalized slope (or Lelong number) at zero, but this Lelong number depends on the choice of a trivialization. A nice way of capturing all possible such Lelong numbers is by looking directly at the metric $(\langle \phi^{d+1}\rangle_{X/\pdisc})\na$ on $\langle L_\K^{d+1}\rangle$! The Lelong number of $u$ specifically is then recovered as the difference of Deligne pairings $(\langle \phi^{d+1}\rangle_{X/\pdisc})\na-\langle \phi_\cL^{d+1}\rangle$, where $\phi_\cL$ is the model metric associated to $\cL$ on $L_\K\an$.
\end{remark}


\begin{theorem}\label{thm_delignelelong}
For all $\phi\in\cE^1(L)$ admitting a locally bounded extension to some model $(\cX,\cL)$, we have
$$(\langle \phi^{d+1}\rangle_{X/\pdisc})\na=\langle(\phi\na)^{d+1}\rangle,$$
as non-Archimedean metrics on the Deligne pairing $\langle L\an \rangle$ over $\Spec \K$.
\end{theorem}
\begin{proof}
Note that the metric $\langle \phi^{d+1}\rangle_{X/\pdisc}$ is subharmonic by \eqref{eq_ddc}, so that the left-hand side is well-defined (this is the relative dimension zero case of Example \ref{exa_reldim2}).

\bsni We pick a model $(\cX,\cL)$ such that $\phi$ extends to a locally bounded metric on $\cL$. By Proposition \ref{prop_extension_on_model}, we necessarily have $\phi\na=\phi_\cL$, the model metric on $L_\K\an$ associated to $\cL$, so that we are left to show that, given a trivialization $\tau$ of $\langle \cL^{d+1}\rangle_{\cX/\disc}$ and setting $u(z)=-\log|\tau(z)|_{\langle \phi_z^{d+1}\rangle}$, we have
$$\hat u=0$$
(recall how we defined the model metric $\phi_\cL$ in Section \ref{sect_nappt}). But $\phi$ is locally bounded near the central fibre of $\cL$, so that $u$ is locally bounded near zero, which implies $\hat u = 0$ as desired.
\end{proof}

\begin{remark}\label{rem_functionals}
We will occasionally refer to a metric satisfying the statement of Theorem \ref{thm_delignelelong} as satisfying the Monge-Ampère extension property. We also remark that the proof of the Theorem works more generally for arbitrary Deligne pairings(!): given $d+1$ pairs of relatively ample line bundles $L_i$ on $X$ and metrics $\phi_i\in \cE^1(L_i)$ admitting locally bounded extensions to some model of $L_i$, one has
$$(\langle \phi_0,\dots,\phi_d\rangle_{X/\pdisc})\na=\langle\phi\na_0,\dots,\phi\na_d\rangle.$$
The fact that the slopes are well-defined follows as in the proof of the above Theorem from the general property (\ref{eq_ddc}) of Deligne pairings! In Section \ref{subsect_35}, we will show how to extend this result to the class of metrics satisfying the Monge-Ampère extension property.
\end{remark}


\subsection{Hybrid maximal metrics: existence and uniqueness.}\label{subsect_33}

We now study hybrid maximal metrics. Such metrics can be described as being relatively maximal, but with boundary values prescribed both at the complex boundary of $X$ and at the "asymptotic" or non-Archimedean boundary. We will then see that they correspond exactly to metrics satisfying the Monge-Ampère extension property.
\begin{defi}Let $\phi\in\hat\cE^1(L)$. We say that $\phi$ is hybrid maximal if for any $\psi\in\cE^1(L)$ such that $\psi\na \leq \phi\na$ and $\limsup(\psi-\phi)\leq 0$ near the boundary of $X$, we have $\psi\leq\phi$.
\end{defi}

\begin{remark}
We show how to relate our terminology with that of \cite{bbj}, which deals with special cases of our objects:
\begin{itemize}
\item a geodesic ray in \cite{bbj} is a relatively maximal $\bbc^*$-invariant (logarithmic growth) psh metric on a line bundle over a test configuration in our article;
\item a maximal geodesic ray in \cite{bbj} is  a hybrid maximal $\bbc^*$-invariant (logarithmic growth) psh metric on a line bundle over a test configuration in our article.
\end{itemize}
The "hybrid" refers to (e.g.) the work of Boucksom-Jonsson, in which a hybrid property is a property that passes well from the complex setting to the non-Archimedean limit. Other possible denominations could be "Lelong-maximal" or "maximal in the non-Archimedean limit", but both of those seem to focus more on the limit behaviour while we require our metric to also be maximal in the complex world.
\end{remark}

\begin{theorem}\label{thm_uniquenesstropmaxdisc}
For any $\Phi\na\in \cE^1(L\an)$, there exists a unique metric $\phi\in\hat\cE^1(L)$ such that $\phi\na=\Phi\na$.
\end{theorem}
\begin{proof}
The proof is in two parts. We begin with the assumption that the non-Archimedean metric is a model metric, and construct the unique solution via adapted envelope techniques (inspired by \cite[Proposition 2.7]{bermank}). Then, for the general case, we use properties of the Monge-Ampère energy.

\bsni \textit{First step: the model case.} Assume thus $\Phi\na$ to be a model metric corresponding to a model $(\cX,\cL)$ of $(X,L)$. Denote by $\phi$ the "Perron-Bremmermann-Lelong" envelope defined as the supremum of all metrics $\psi\in\cE^1(L)$ with $$\lim_{z\to\xi}\psi(z)\leq \phi_z$$
for all $\xi\in \partial X$, and
$$\psi\na\leq\phi\na.$$
We begin with a claim that $\phi$ so defined belongs to $\cE^1(L)$. Note that if we can show that it is plurisubharmonic, then it necessarily has logarithmic growth, as the supremum of a family of metrics with logarithmic growth, and it is by definition relatively maximal. Furthermore, the fibrewise finite-energy condition will also immediately follow, so that we need to focus on the plurisubharmonicity. Let $\psi$ be in the class of contributions to the supremum above. The hypothesis that $\psi\na\leq \Phi\na$ implies via Proposition  \ref{prop_extension_on_model} that $\psi$ extends with at worst analytic singularities as a psh metric on $\cL$. We therefore see $\phi$ to be the restriction of a metric $\phi_\cX$ defined as the supremum of all metrics on $\cL$, with the same boundary conditions as above on $\partial \cX$, and extending with at worst analytic singularities over the central fibre of $\cX$. That the envelope satisfies our claim is then a particular case of Theorem \ref{thm_existencemaximal} (which allows singular fibres!).

\bsni Finally, the second case of Proposition \ref{prop_extension_on_model} together with the non-Archimedean maximality assumption ensure that it is hybrid maximal, provided we can show that for any model metric $\Phi\na$ there exists a metric $\psi\na\in\cE^1(L)$ with $\psi\na=\Phi\na$. But this also follows from the same Lemma, since one only has to choose $\psi$ to be a psh metric with a locally bounded extension to $\cL$. That $\phi$ is the unique hybrid metric given our data follows again from the extremal characterization.

\bsni \textit{Second step: the general case.} The general case again proceeds by approximation: we pick a sequence of model metrics $\phi_i\na$ decreasing to $\Phi\na$, and their associated hybrid maximal metrics $\phi_i$ in $\hat\cE^1(L)$, which exist and are unique by the first part of the proof. The $\phi_i$ then give a decreasing sequence of metrics by maximality. We write $\phi$ for their limit and $\phi\na$ the non-Archimedean metric it defines. Since the mappings $\phi\mapsto \phi\na$ are order-preserving, we find
$$\phi\na\leq\phi_k\na$$
for all $k$, i.e.
\begin{equation}\label{eq_harm1}
\phi\na\leq\Phi\na.
\end{equation}
Fix a model $(\cX,\cL)$ of $(X,L)$, so that $\phi$ and the $\phi_k$ extend to $\cL$ (with singularities). Fix a trivialization $\tau$ of $\langle \cL^{d+1}\rangle_{\cX/\disc}$, and set
$$E(\phi_z):=-\log|\tau(z)|_{\langle \phi_z^{d+1}\rangle},$$
$$E(\phi_{k,z}):=-\log|\tau(z)|_{\langle \phi_{k,z}^{d+1}\rangle}.$$
We will also denote as usual by $E\na(\phi\na)$ the metric $\langle (\phi\na)^{d+1}\rangle$. Now, by Corollary \ref{coro_trivialization}, and Theorem \ref{thm_delignelelong}, we have for all $k$
\begin{equation}\label{eq_harm2}
E(\phi_{k,z})=c_k\cdot \log|z| + H(z),
\end{equation}
where $H=H(\phi_\partial)$ is some function bounded near zero and independent of $k$. In fact, one can see that
$$c_k=-(\langle (\phi_k\na)^{d+1}\rangle - \langle \phi_\cL^{d+1}\rangle)=-(E\na(\phi_k\na)-E\na(\phi_\cL)).$$
Since the (Archimedean and non-Archimedean) Monge-Ampère energies are continuous along decreasing nets, we have
$$E\na(\phi_k\na)\to E\na(\phi\na)$$
while
$$E(\phi_{k,z})\to E(\tilde\phi_z)$$
for all $z$. Combining those with (\ref{eq_harm2}), one finds
\begin{equation}\label{eq_harm3}
E(\phi_z)=-(E\na(\phi_k\na)-E\na(\phi_\cL))\cdot \log|z| + H(z),
\end{equation}
which by Corollary \ref{coro_trivialization} shows that $\phi$ is a relatively maximal metric. Furthermore, we know that $\PSH(L)$ is closed under decreasing limits: $\phi$ thus has logarithmic growth. To establish existence, i.e. to show that $\phi$ is our desired solution, we now only have to show that $\phi\na=\Phi\na$. Using \cite[Proposition 6.3.2]{reb20b}, this is proven provided we can show that
\begin{equation}\label{eq_harm4}
E\na(\phi\na)=E\na(\Phi\na)
\end{equation}
by (\ref{eq_harm1}). One inequality is immediate from the same equation (\ref{eq_harm1}) and monotonicity of $E\na$:
$$E\na(\phi\na)\leq E\na(\Phi\na).$$
From (\ref{eq_harm3}) we have
\begin{equation}\label{eq_harm5}
\hat E(\phi)=E\na(\Phi\na)-E\na(\phi_\cL),
\end{equation}
so that we have the other inequality (hence (\ref{eq_harm4})), provided we can show that
\begin{equation}\label{eq_harm6}
\hat E(\phi)\leq E\na(\phi\na)-E\na(\phi_\cL).
\end{equation}
This inequality follows from a similar argument. Let $\psi\na_k$ be a decreasing sequence of model metrics approximating $\phi\na$. Let $\psi_k$ denote their associated hybrid maximal metric, and define $E(\psi_k)$ as before. Now, since for all $k$ $\phi\na\leq \psi\na_k$, by maximality, we have
\begin{equation}
\phi\leq \psi_k
\end{equation}
whence
\begin{equation}
E(\phi_z)\leq E(\psi_{k,z}).
\end{equation}
Taking negative Lelong numbers,
\begin{equation}
\hat E(\phi)\leq \hat E(\psi_{k}).
\end{equation}
By Theorem \ref{thm_delignelelong} and the arguments above, $\hat E(\psi_{k})=E\na(\psi_k\na)-E\na(\phi_\cL)$ which again upon taking the decreasing limit in the right-hand side (along which $E\na$ is continuous) establishes
\begin{equation}
\hat E(\phi)\leq \lim_k E\na(\psi\na_{k}))-E\na(\phi_\cL)=E\na(\phi\na)-E\na(\phi_\cL)
\end{equation}
by definition of the $\psi_k\na$. This establishes (\ref{eq_harm6}) as desired, hence existence of a hybrid maximal metric with non-Archimedean metric equal to $\Phi\na$. That such a segment is unique is then a consequence of the extremal definition of hybrid maximality.
\end{proof}

\begin{corollary}\label{coro_eonemapping}
The space $\cE^1(L)$ is mapped by $(\cdot)\na$ to $\cE^1(L_\K\an)$; furthermore, for any $\phi\in\cE^1(L)$, we have
$$(\langle \phi^{d+1}\rangle_{X/\pdisc})\na\leq \langle (\phi\na)^{d+1}\rangle.$$
In other words, we do not have non-Archimedean extension of the Monge-Ampère energy in $\cE^1(L)$, but simply an inequality.
\end{corollary}
\begin{proof}
We start by picking a metric $\phi\in\cE^1(L)$, and we define $\Phi$ to be the hybrid maximal metric with $\Phi\na=\phi\na$ obtained from Theorem \ref{thm_uniquenesstropmaxdisc}. Then, since $\Phi$ is relatively maximal, $\phi\leq \Phi$, so that by monotonicity of $\phi\mapsto\phi\na$,
$$(\langle \phi^{d+1}\rangle_{X/\pdisc})\na\leq (\langle \Phi^{d+1}\rangle_{X/\pdisc})\na$$
while $(\langle \Phi^{d+1}\rangle_{X/\pdisc})\na=\langle(\Phi\na)^{d+1}\rangle$ by hybrid maximality, proving our inequality. To prove the first statement, it is enough to notice that the logarithmic growth condition built into $\cE^1(L)$ forces $(\langle \phi^{d+1}\rangle_{X/\pdisc})\na$ to be finite.
\end{proof}

\subsection{The isometric embedding.}\label{subsect_34}

We denote by $$\hat\cE^1_{\mathrm{hyb}}(L)$$ the subspace of hybrid maximal metrics in $\hat\cE^1(L)$. Our main Theorem is the following:
\begin{theorem}\label{thm_isometry}
The inverse of the mapping $(\cdot)\mapsto(\cdot)\na$ given by Theorem \ref{thm_uniquenesstropmaxdisc} is an isometric embedding of $(\cE^1(L\an),d_1\na)$ into $(\hat\cE^1(L),\hat d_1)$ with image $\hat\cE^1_{\mathrm{hyb}}(L)$. Furthermore, a psh segment in $\cE^1(L\an)$ is a psh geodesic if and only if its image is a psh geodesic.
\end{theorem}
\begin{remark} The first statement of the Theorem can be thought of as saying that hybrid maximal metrics have the $d_1$-extension property. The whole of Theorem \ref{thm_isometry} essentially means that we realize the (non-Archimedean) space $\cE^1(L\an)$ as a purely complex geometric object!
\end{remark}

\begin{proof}
We thus first show that our mapping preserves psh geodesic segments. Pick a psh geodesic segment $\phi_t\na$ between two metrics $\phi_0\na$ and $\phi_1\na$ in $\cE^1(L_\K\an)$, and consider for all $t\in[0,1]$ the hybrid maximal metric $\phi_t$ with $(\phi_t)\na=\phi_t\na$. By Theorem \ref{thm_1geodesics}, it is enough to show that $t\mapsto (\langle \phi_t^{d+1}\rangle_{X/\pdisc})\na$ is affine; by the Monge-Ampère energy extension property, this is equivalent to asking that $t\mapsto\langle \phi_t\na\rangle$ is affine, which holds by \cite{reb20b}. The reverse implication is proved in the same way.

\bsni We now prove the isometry statement. Pick $\phi_0$, $\phi_1$ in $\hat\cE^1_{\mathrm{hyb}}(L)$. We assume both metrics to be continuous, and the general result will proceed as usual from regularization. Using Theorem \ref{thm_uniquenesstropmaxdisc} together with the expressions of the distances and additivity of Lelong numbers,
$$d_1\na(\phi_0\na,\phi_1\na)=\langle(\phi_0\na)^{d+1}\rangle+\langle(\phi_1\na)^{d+1}\rangle-2\langle(P(\phi_0\na,\phi_1\na)^{d+1}\rangle,$$
$$d_1(\phi_{0,z},\phi_{1,z})=\langle\phi_{0,z}^{d+1}\rangle+\langle\phi_{1,z}^{d+1}\rangle-2\langle P(\phi_{0,z},\phi_{1,z})^{d+1}\rangle,$$
we only have to show that
$$(-\langle P(\phi_{0},\phi_{1})^{d+1}\rangle_{X/\pdisc})\na=-\langle P(\phi_0\na,\phi_1\na)^{d+1}\rangle.$$
Recall that we have seen $z\mapsto \langle P(\phi_{0,z},\phi_{1,z})^{d+1}\rangle$ to be superharmonic, so that the left-hand side is well-defined.

\bsni Set some $r\in (0,1)$. We consider the relatively maximal metric $\psi_r$ on the preimage $U_r$ of the annulus $\{r\leq z \leq 1\}$ with boundary data given by $z\mapsto P(\phi_{0,z},\phi_{1,z})$ for $z\in\partial U_r$, which exists by Theorem \ref{thm_existencemaximal}. Having fixed $z\in X$, the sequence $r\mapsto \psi_{r,z}$, $r\leq |\pi(z)|$, is decreasing as $r$ decreases, and therefore the limit $\lim_{r\to 0}\psi_r=:\psi$ is still a relatively maximal metric. As we have, for all $|\pi(z)|=r$,
$$\langle \psi_{r,z}^{d+1}\rangle=\langle P(\phi_{0,z},\phi_{1,z})^{d+1}\rangle,$$
it follows that
$$-(\langle \psi^{d+1}\rangle_{X/\pdisc})\na=(-\langle P(\phi_{0},\phi_{1})^{d+1}\rangle_{X/\pdisc})\na.$$
We must now prove that
$$(\langle \psi^{d+1}\rangle_{X/\pdisc})\na=\langle P(\phi_0\na,\phi_1\na)^{d+1}\rangle.$$
We first claim that $\psi$ realizes the supremum
$$\psi=\sup\{\varphi\in\PSH(L),\,\varphi\leq\phi_0,\phi_1\}.$$
Since $\psi$ is itself such a metric, it is enough to show that for all candidates $\varphi$, we have $\varphi\leq\psi$. But for all $z\in X$, since $\varphi_z\leq \phi_{0,z},\phi_{1,z}$, we have
$$\varphi_z\leq P(\phi_{0,z},\phi_{1,z})$$
hence
$$\varphi_z\leq \psi_{r,z}$$
and finally
$$\varphi_z\leq \lim_r \psi_{r,z}=\psi_z.$$
We now conclude: by this extremal characterization of $\psi$, we have that $\varphi\na\leq \psi\na$ for all $\varphi\leq \phi_0,\phi_1$. In particular, since the construction is order-preserving, the hybrid maximal metric $\Psi$ with $\Psi\na=P(\phi_0\na,\phi_1\na)$ satisfies $\Psi\leq \psi$, so that
$$P(\phi_0\na,\phi_1\na)\leq \psi\na,$$
while on the other hand, $\psi\leq \phi_0,\phi_1$, hence $\psi\na\leq \phi_0\na,\phi_1\na$ and finally
$$\psi\na\leq P(\phi_0\na,\phi_1\na).$$
\end{proof}

\begin{remark}The proof of the above result in the case of geodesic rays, which does not appear explicitly in the literature (but is based on some ideas from \cite{bdl}), was nicely explained to the author by Tamas Darvas.
\end{remark}

\begin{remark}In the above proof, we implicitly defined an envelope operator sending two metrics $\phi_0$, $\phi_1$ in $\hat\cE^1_{\mathrm{hyb}}(L)$ to the largest metric $\hat P(\phi_0,\phi_1)$  in $\hat\cE^1_{\mathrm{hyb}}(L)$ bounded above by $\phi_0$ and $\phi_1$. In \cite[Example 3.3]{xia}, this construction appears already in the case of geodesic rays, and Xia uses this envelope to define
alternative distance
$$\hat d_1'(\phi_0,\phi_1):=\lim_t t\mi (E(\phi_{0,t})+E(\phi_{1,t})-2E(\hat P(\phi_0,\phi_1)_t),$$
which (a specialization of) our proof shows to coincide with the usual distance $\hat d_1$. In fact, Xia defines this envelope more generally in \cite[Example 3.2]{xia}, in the radial equivalent of the space $\hat\cE^1(L)$. It is likely that this construction generalizes to metrics in $\hat\cE^1(L)$ in our setting, although this is outside the scope of the present article.
\end{remark}

\subsection{Non-Archimedean extension of generalized functionals.}\label{subsect_35}

In Theorem \ref{thm_eonepairing}, we have seen that the fibrewise finite energy condition is the adequate condition for finiteness of fibrewise Deligne pairings. Further following the mantra that properties pertaining to the energy govern the same properties for more general Deligne pairings, we show that non-Archimedean extension of generalized energy functionals in the sense of Remark \ref{rem_functionals} holds for our class of hybrid maximal metrics, i.e. metrics satisfying the Monge-Ampère extension property.

\begin{prop}\label{prop_functionals}Suppose given $d+1$ relatively ample line bundles $L_i$ on $X$. Then, for any $(d+1)$-uple of metrics $\phi_i\in\hat\cE^1_{\mathrm{hyb}}(L_i)$, we have
$$(\langle\phi_0,\dots,\phi_d\rangle_{X/\pdisc})\na=\langle\phi\na_0,\dots,\phi\na_d\rangle.$$
\end{prop}
\begin{proof}
We approximate each of the $\phi_i\na$ by a decreasing sequence of model metrics $\phi_{i,k}\na$, and denote by $\phi_{i,k}$ their associated hybrid maximal metrics. By our previous results, $\phi_{i,k}$ decreases to $\phi_i$ by hybrid maximality. Since Deligne pairings are decreasing along (mixed) decreasing limits, using the estimates \cite[Lemma A.1, Lemma A.2]{bbj}, we find for all $z$ in $X$
$$0\leq \langle \phi_{0,k,z},\dots,\phi_{d,k,z}\rangle-\langle\phi_{0,z},\dots,\phi_{d,z}\rangle\leq C(z)\cdot \max_i d_1(\phi_{i,k,z},\phi_{i,z}),$$
where the slope of $C(z)$, $\hat C$, is a finite real constant. Indeed, $C(z)$ is a maximum of a collection of functionals expressed as Deligne pairings, which are all subharmonic along relatively maximal metrics. We take generalized slopes in the above inequality to find
$$0\leq (\langle\phi_{0,k},\dots,\phi_{d,k}\rangle_{X/\pdisc})\na-(\langle\phi_{0},\dots,\phi_{d}\rangle_{X/\pdisc})\na\leq \hat C\cdot \max_i d_1\na(\phi_{i,k}\na,\phi_{i}\na),$$
where we have used the $d_1$-extension property of hybrid maximal metrics. Using Remark \ref{rem_functionals}, we then have that
$$0\leq \langle\phi_{0,k}\na,\dots,\phi_{d,k}\na\rangle-(\langle\phi_{0},\dots,\phi_{d}\rangle_{X/\pdisc})\na\leq \hat C\cdot \max_i d_1\na(\phi_{i,k}\na,\phi_{i}\na),$$
and taking the limit in $k$ in the above inequality finally gives our result.
\end{proof}

\begin{example}\label{exa_functionals}Many functionals acting on $\PSH(L)$ satisfy the statement of the above Proposition. Having fixed some reference metric $\refmetric\in\hat\cE^1_{\mathrm{hyb}}(L)$, some among the most important are:
\begin{enumerate}
\item the $I$-functional, which appeared in the estimates mentioned in the above proof, is defined as
$$I(\phi)=\langle \phi-\refmetric,\refmetric^d\rangle_{X/\pdisc} - \langle \phi-\refmetric,\phi^d\rangle_{X/\pdisc},$$
which has many important norm-like properties and is commonly used to study properties of finite-energy spaces (\cite{bbegz}, \cite{bhjasymptotics}, see also \cite{bjtrivval21} for the non-Archimedean trivially-valued case);
\item the $J$-functional, defined as
$$J(\phi)=\langle\phi,\refmetric^d\rangle_{X/\pdisc}-\langle\refmetric^{d+1}\rangle_{X/\pdisc}-(d+1)\mi(E(\phi)-E(\refmetric)),$$
which can be seen as a corrected relative Monge-Ampère energy which is translation invariant;
\item the twisted energy functionals, defined as
$$E^{\psi}(\phi)=\langle \psi,\phi^d\rangle_{X/\pdisc},$$
for $\psi\in\hat\cE^1_{\mathrm{hyb}}(L')$, where $L'$ is another line bundle on $X$. A special case of it appears in the expression of the Mabuchi K-energy, and the study of its slopes in the trivially-valued case is essential to establish the general (cscK) case of the Yau-Tian-Donaldson conjecture, as in \cite{liytd}.
\end{enumerate}
\end{example}

\subsection{Test configurations and the trivially valued case.}\label{subsect_36}

All of our previous results encapsulate the trivially valued case, as we explain now. Let $\pi:X\to\bbd^*$ be now a polarized test configuration, i.e. a degeneration with relatively ample line bundle $L$ such that $\pi$ and $L$ are equivariant under some $\bbc^*$-action (forcing all fibre pairs $(X_z,L_z)$ to be isomorphic). One may then choose a reference continuous psh metric $\refmetric$ on the fibre at $1$ and require our psh metrics $\phi$ to satisfy $\phi_z=i_z{}^*\refmetric$ for $z\in\bbs^1$, and with
$$i_z:X_z\to X_1$$
the isomorphism as mentioned above. The authors in \cite{bbj} (e.g.) study the space $\cE^1_0(X_1\an)$ of finite-energy metrics over the analytification of $X_1$ with respect to the trivial absolute value on $\bbc$. We denote by $\cR^1(L_1)$ the space of hybrid maximal finite-energy rays in $\PSH(L_1)$ emanating from $\refmetric$ (where, as mentioned before, a hybrid maximal ray corresponds in the terminology of \cite{bbj} to a maximal psh geodesic ray). We then claim the following:
\begin{prop}
There is a sequence of distance-preserving maps
$$\cE^1_0(L_1\an)\simeq \cR^1(L_1) \hookrightarrow \hat\cE_{\mathrm{hyb}}^1(L) \simeq \cE^1(L\an),$$
where the first and last maps are bijective (i.e. isometries), while the middle map is injective.
\end{prop}
\begin{proof}
The case of the last map has been treated by Theorem \ref{thm_isometry}. The rest of the proof is merely a matter of correctly defining our maps.

\bsni For the first map, the bijection is given by \cite[Theorem 6.6]{bbj}. The metrization of the space $\cE^1_0(L_1\an)$ is described in a \cite{bjtrivval21}, but proceeds much as the metrization of $\cE^1(L\an)$ in \cite{reb20b}, while we recall that we metrize the space of maximal rays by 
$$\hat d_{1,0}(\phi,\phi')=\lim_t t\mi d_1(\phi_t,\phi'_t)$$
and take equivalence classes to yield the space $\cR^1(L_1)$.
(We direct the reader to e.g. \cite{bdl}. Note that in the cited article, the authors consider the space of all (non-necessarily hybrid) maximal psh rays.) Proving the distance-preservingness of the isomorphism is then essentially a simpler version of Theorem \ref{thm_isometry}, which we leave to the interested reader.

\bsni We claim that the middle map, which we will denote $\iota_0$, can be represented as follows: let $\phi:t\mapsto \phi_t$ be a hybrid maximal psh geodesic ray in $X_1$. Let $i_z$ be as before the isomorphism $i_z:X_z\to X_1$, and define $\iota_0(\phi)$ to be the metric
$$z\mapsto i_z{}^*(\phi_{-\log|z|}).$$
The distance-preservingness is immediate(!), so that we are left to check that $\iota_0(\phi)$ is a hybrid maximal metric. By \cite[Corollary 6.7]{bbj},
$$t\mapsto E(\phi_t)$$
is affine, which implies by invariance of the energy under polarized isomorphisms that
$$z\mapsto \iota_0(\phi)(z)$$
is harmonic on $\bbd^*$, proving maximality by Proposition \ref{prop_harmonicitymaximaldiscs}, and hybrid maximality is given by construction.
\end{proof}

\begin{remark}
One also notices (by mimicking our proofs in the discretely valued case) there to be under the above maps a correspondance 
$$\{\text{non-Archimedean maximal psh segments in }\cE^1_0(L_1\an)\}$$
$$\mathrel{\rotatebox[origin=c]{-90}{$\simeq$}}$$
$$\{\text{rays of complex geodesics between two rays in }\cR^1(L_1)\}$$
$$\mathrel{\rotatebox[origin=c]{-90}{$\hookrightarrow$}}$$
$$\{\text{discs of complex geodesics between two metrics in }\hat\cE_{\mathrm{hyb}}^1(L)\}$$
$$\mathrel{\rotatebox[origin=c]{-90}{$\simeq$}}$$
$$\{\text{non-Archimedean maximal psh segments in }\cE^1(L\an)\}.$$
We may state the most interesting part of this result as follows.
\end{remark}

\begin{prop}\label{prop_trivval}
Let $t\mapsto \phi_{0,t}$, $t\mapsto \phi_{1,t}$ be maximal psh geodesic rays in the sense of \cite{bbj}. Let, for all $t$,
$$[0,1]\ni s\mapsto \phi_{s,t}$$
be the maximal psh segment joining $\phi_{0,t}$ and $\phi_{1,t}$. Then, for all $s\in[0,1]$, $t\mapsto \phi_{s,t}$ is a maximal psh geodesic ray in the sense of \cite{bbj}.

\bsni Furthermore, let for all $s\in[0,1]$, $\phi_s\na$ be the non-Archimedean metric associated to the psh geodesic ray $s\mapsto \phi_{s,t}$. Then,
$$s\mapsto \phi_s\na$$
is the maximal non-Archimedean psh geodesic joining $\phi_0\na$ and $\phi_1\na$ in the sense of \cite{reb20b}.
\end{prop}

\subsection{Convexity of non-Archimedean functionals.}\label{subsect_37}

Via the isometry $\iota$ given by Theorem \ref{thm_isometry}, it is now clear what we mean by "a functional on the space of hybrid maximal metrics", since $\hat\cE_{\mathrm{hyb}}^1(L)$ inherits a $\K$-vector space structure by setting, for all $\phi$, $\psi\in\hat\cE^1(L)$, and $\lambda\in\K$,
$$\phi+\lambda\cdot\psi=\iota\mi(\iota(\phi)+\lambda\cdot\iota(\psi)).$$
(In particular, one can see multiplication by a scalar in $\K$ as a fibrewise scaling of the metrics, varying meromorphically.)

\begin{defi}
Let $\hat F$ be a functional on $\hat\cE_{\mathrm{hyb}}^1(L)$ and $F\na$ be a functional on $\cE^1(L\an)$. We say that $F\na$ is a non-Archimedean extension of $\hat F$ if the diagram
\begin{center}
\begin{tikzcd}[column sep=small]
{\hat\cE^1(L)} \arrow["\iota", rr] \arrow["\hat F" swap, rd] &    & {\cE^1(L\an)} \arrow["F\na", ld] \\
                         & {\bbr} &              
\end{tikzcd}
\end{center}
commutes.
\end{defi}

\begin{example}By construction, $E\na$ is a non-Archimedean extension of the "discal" energy $\hat E$. Furthermore, all "generalized energy" functionals of Proposition \ref{prop_functionals} and Example \ref{exa_functionals} admit non-Archimedean extensions.
\end{example}

\noindent By Theorem \ref{thm_isometry}, we get "for free" a way to study convexity of non-Archimedean functionals.

\begin{heuristic}\label{heur_convexity}Let $F$ be a functional that is  convex along complex psh geodesics, $\hat F$ its "discal" version, and $F\na$ a non-Archimedean extension of $\hat F$. Then $F\na$ is convex along maximal non-Archimedean psh geodesics in $\cE^1(L\an)$.
\end{heuristic}
\begin{proof}
Given a non-Archimedean maximal psh segment $\phi\na_t$ joining $\phi\na_0$ and $\phi\na_1\in\cE^1(L\an)$, we can write using Theorem \ref{thm_isometry} $\phi_t:=\iota\mi(\phi\na_t)$ using the maximal psh segments $t\mapsto \phi_{t,z}$ joining the $\phi_{0,z}$ and $\phi_{1,z}$. We then simply write for all $z$
$$F(\phi_{t,z})\leq (1+t)F(\phi_{0,z})+tF(\phi_{1,z}),$$
and take the limit to find
$$\hat F(\phi_t)\leq (1+t)\hat F(\phi_0)+t\hat F(\phi_1),$$
which by the definition of a non-Archimedean extension together with $\iota(\iota\mi(\phi\na))=\phi\na)$ gives
$$F\na(\phi\na_t)\leq (1+t)F\na(\phi\na_0)+tF\na(\phi\na_1)$$
as desired.
\end{proof}

\begin{remark}
Using Proposition \ref{prop_trivval}, the same results also hold \textit{mutatis mutandis} in the trivially-valued case.
\end{remark}

\begin{example}\label{exa_mabuchi}This allows us to obtain convexity of the (trivially-valued) non-Archimedean K-energy modulo the entropy approximation conjecture, as follows. Pick a compact Kähler manifold $X_1$ together with an ample line bundle $L_1$, and consider the trivially-valued analytification $(X_1\an,L_1\an)$ as before. One then introduces the non-Archimedean entropy as
$$H\na(\phi\na)=\int_X A_X \MA(\phi\na),$$
where $A_X$ is the log-discrepancy function on $X\an$ identified with a space of semivaluations (\cite{bjkstab}), and $\phi\in\cE^1_0(L_1\an)$. The entropy approximation conjecture states that, given $\phi\na\in\cE^1_0(L_1\an)$, there exists a sequence $\phi_k\na$ of model metrics converging to $\phi\na$ such that $H\na(\phi_k\na)\to H\na(\phi)$.

\bsni Now, Chi Li (\cite[Conjecture 1.6]{liytd}) shows that, assuming this conjecture, the entropy $H\na$ is exactly the non-Archimedean (radial) extension of the usual complex entropy functional. Adding the energy part that makes up the Mabuchi K-energy, which is convex along complex geodesics, and using Example \ref{exa_functionals}, our previous result shows that the non-Archimedean K-energy is convex along non-Archimedean geodesics if the conjecture holds.

\bsni As things currently stand, extension of the K-energy is only known for rays that generate model metrics, via \cite[Theorem 3.6]{bhjasymptotics}. However, the non-Archimedean geodesics of \cite{reb20b} do not remain in the space of model metrics even if the endpoints are, much as geodesics between Kähler potentials are merely $C^{1,\bar 1}$.
\end{example}

\begin{example}\label{exa_ding}In \cite{blumliu}, Blum-Liu-Xu-Zhuang prove, using algebraic techniques, convexity of the non-Archimedean Ding energy (and other functionals) along geodesics between test configurations (\cite[Theorem 3.7]{blumliu}). Our heuristic allows us to also recover this result: convexity of the complex Ding energy is a result of Berndtsson (\cite{berndtding}) while the non-Archimedean extension of the Ding energy follows from \cite{bhjasymptotics}.
\end{example}

\begin{remark}As a Corollary, existence and uniqueness of minimizers for non-Archimedean energy functionals can also be detected strict convexity results in the complex setting, but some form of  uniform strict convexity (in $|z|$) is required to ensure strict convexity of the limit.
\end{remark}

\subsection{Kähler-Einstein metrics in families.}\label{subsect_misc}

Let $M$ be a complex projective manifold with ample canonical bundle $K_M$. It is a consequence of the by now classical Aubin-Yau Theorem that $M$ carries a Kähler-Einstein metric. If $X$ is more generally a family of canonically polarized projective manifolds, the family $z\mapsto\phi_z$ of fibrewise Kähler-Einstein metrics is known to have plurisubharmonic variation (from the work of e.g. Schumacher \cite{schumacher}) and to have logarithmic growth (\cite[Theorem 3]{sch08}). In particular, $\phi$ defines a metric in our class $\cE^1(K_X)$. Interpreting the work of Pille-Schneider through our lens, one can see \cite[Theorem A]{ps} to essentially imply that $\phi\na$ corresponds to a model metric in $\cH(K_X\an)$ associated to a distinguished model of $(X,K_X)$ (see e.g. \cite{tian}, \cite{song}).

\bsni An immediate question arises: how does the metric $\phi$ relate to our relatively maximal metrics framework? In particular, how does $\phi$ relate to the hybrid maximal metric $\Phi$ corresponding to $\phi\na$? Interestingly, $\phi$ is not even relatively maximal when the Kodaira-Spencer class of the family $X$ is nontrivial, by \cite[Main Theorem]{schumacher}, since $\phi$ will be strictly positive (in particular, cannot satisfy $\MA(\phi)=0$). As a consequence, we have that $dd^c d_1(\phi,\Phi)=dd^c (E(\Phi) - E(\phi))$  is given explicitly by the formula of Schumacher, using the pushforward formula for Deligne pairings.

\bsni Naturally, it would be interesting to know whether one could detect via non-Archimedean tools the existence of a family of Kähler-Einstein metrics in the class of a hybrid maximal metric. This seems a bit ambitious, since one only captures the "asymptotic" behaviour of a family of metrics when considering non-Archimedean data. A more realistic (and perhaps just as interesting) problem would be solving the following hybrid "almost Kähler-Einstein" problem: to find $\phi\in\hat\cE^1_{\mathrm{hyb}}(K_X)$, such that $\phi\na=\psi\na$ where $\psi\na$ is an "almost Kähler-Einstein metric":
\begin{align*}
(\omega_z +i\partial\bar\partial \psi_z)^d - e^{h_{\omega_z} + \psi_z}\omega_z{}^d \to_{z\to 0} 0.
\end{align*}
The upshot is that this problem gives, intuitively, a purely non-Archimedean criterion for the existence of a family of complex manifolds degenerating to a Kähler-Einstein manifolds! (Of course, the same problem arises in the (possibly twisted) Fano case.)


\bsni Finally, we briefly mention an additional difficulty in the Calabi-Yau case. By a counterexample of Cao-Guenancia-Paun, we know that a family $\phi=(\phi_z)_z$ of Kähler-Einstein metrics on a degeneration of Calabi-Yau manifolds does not necessarily vary plurisubharmonically (\cite[Theorem 3.1]{cgp2}). One can however take the plurisubharmonic envelope $P(\phi)$ of $\phi$, and then the hybrid maximal metric $\Phi$ with $\Phi\na=P(\phi)\na$. In \cite{bjtrop}, Boucksom-Jonsson show that the family of measures $\MA(\phi_z)$ converge in a certain sense to the non-Archimedean Monge-Ampère measure of some metric $\psi\na$. We therefore formulate the following result, which would connect our hybrid maximal setting with degenerations of Kähler-Einstein metrics on Calabi-Yau manifolds:
\begin{conjecture}
$\psi\na=P(\phi)\na=\Phi\na$.
\end{conjecture}

\newpage
\bibliographystyle{alpha}
\newcommand{\etalchar}[1]{$^{#1}$}

\end{document}